\documentclass[11pt, reqno]{amsart}
\usepackage[british]{babel}

\usepackage{amsmath, amsfonts, amsthm, amssymb, amscd, mathtools, multicol}

\usepackage{graphicx}

\usepackage{color}
\usepackage[usenames,dvipsnames,table,xcdraw]{xcolor}
\usepackage[colorlinks=true]{hyperref}
\usepackage{float}
\usepackage{comment}

\allowdisplaybreaks

\numberwithin{equation}{section}

\newtheorem{thm}{Theorem}[section]
\newtheorem{lma}[thm]{Lemma}
\newtheorem{cor}[thm]{Corollary}
\newtheorem{defn}[thm]{Definition}

\newtheorem{prop}[thm]{Proposition}

\newtheorem{rem}[thm]{Remark}
\newtheorem{ques}[thm]{Question}

\newtheorem{example}[thm]{Example}

\newcommand{\ii}{\mathbf{i}}
\newcommand{\jj}{\mathbf{j}}

\renewcommand{\epsilon}{\varepsilon}

\renewcommand{\geq}{\geqslant}
\renewcommand{\leq}{\leqslant}

\DeclareMathOperator*{\argmin}{arg\,min}
\DeclareMathOperator*{\argmax}{arg\,max}

\usepackage{tikz}
\usepackage{pgffor}
\definecolor{zzttqq}{rgb}{0.6,0.2,0}

\begin{document}

\title[The Assouad dimension of self-affine measures on sponges]{The Assouad dimension of self-affine measures on sponges}
\author{Jonathan M. Fraser and Istv\'an Kolossv\'ary}
\address{University of St Andrews, School of Mathematics and Statistics,
	 \newline  St Andrews, KY16 9SS, Scotland}
\email{ jmf32@st-andrews.ac.uk and  itk1@st-andrews.ac.uk}

\begin{abstract}

We derive upper and lower bounds for the Assouad and lower dimensions of self-affine measures in $\mathbb{R}^d$ generated by diagonal matrices and satisfying suitable separation conditions. The upper and lower bounds always coincide for $d=2,3$ yielding precise explicit formulae for the dimensions.  Moreover, there are easy to check conditions guaranteeing that the bounds coincide for $d \geq 4$.

An interesting consequence of our results is that there can be  a `dimension gap' for such self-affine constructions, even in the plane.  That is, we show that for some self-affine carpets of `Bara\'nski type' the Assouad dimension of all associated self-affine measures strictly exceeds the Assouad dimension of the carpet by some fixed $\delta>0$ depending only on the carpet.  We also provide examples of self-affine carpets of `Bara\'nski type' where there is no dimension gap and in fact the Assouad dimension of the carpet is equal to the Assouad dimension of a carefully chosen self-affine measure.\\


\end{abstract}

\maketitle

\section{Introduction: dimensions of self-affine measures}

Let $\nu$ be a compactly supported Borel probability measure in $\mathbb{R}^d$.  The Assouad and lower dimensions of $\nu$ quantify the extremal local fluctuations of the measure by considering the relative measure of concentric balls. In particular, a measure is doubling if and only if it has finite Assouad dimension, e.g. \cite[Lemma 4.1.1]{fraser_2020Book}. Write   $\mathrm{supp}(\nu)$ to denote the support of $\nu$ and $|F|$ to denote  the diameter of a non-empty set $F$. The \emph{Assouad dimension} of $\nu$ is defined by
\begin{multline*}
\dim_{\mathrm{A}} \nu=\inf \bigg\{s \geq 0: \text{ there exists } C>0 \text{ such that, for all }  x \in \mathrm{supp}(\nu) \\
\text{and for all } 0<r<R<|\mathrm{supp}(\nu)|,\;  
\frac{\nu(B(x, R))}{\nu(B(x, r))} \leq C\left(\frac{R}{r}\right)^{s}\bigg\},
\end{multline*}
and, provided $|\mathrm{supp}(\nu)|>0$, the \emph{lower dimension} of $\nu$ is
\begin{multline*}
\dim_{\mathrm{L}} \nu=\sup \bigg\{s \geq 0: \text{ there exists } C>0 \text{ such that, for all } x \in \mathrm{supp}(\nu) \\
\text{and for all } 0<r<R<|\mathrm{supp}(\nu)|,\;  
\frac{\nu(B(x, R))}{\nu(B(x, r))} \geq C\left(\frac{R}{r}\right)^{s}\bigg\}.
\end{multline*}
If $|\mathrm{supp}(\nu)|=0$, then $\dim_{\mathrm{L}} \nu=0$.  The Assouad and lower dimensions of measures were introduced by K\"aenm\"aki,  Lehrb\"ack  and Vuorinen~\cite{KaenmakiLehrbackVuorinen_Indiana13}, where they were originally referred to as the upper and lower regularity dimensions, respectively.  We are interested in the Assouad and lower dimensions of self-affine measures.  

Given a finite index set $\mathcal{I}=\{1,\ldots,N\}$, an affine \emph{iterated function system} (IFS) on $\mathbb{R}^d$ is a finite family $\mathcal{F}=\{f_i\}_{i\in\mathcal{I}}$ of affine contracting maps $f_i(x)=A_i x+t_i$. The IFS determines a unique, non-empty compact set $F$, called the \emph{attractor}, that satisfies the relation
\begin{equation*}
	F = \bigcup_{i\in\mathcal{I}} f_i(F).
\end{equation*}
Given a probability vector $\mathbf{p}=(p(i))_{i\in\mathcal{I}}$ with strictly positive entries, the \emph{self-affine measure} $\nu_{\mathbf{p}}$ fully supported on $F$ is the unique Borel probability measure 
\begin{equation*}
	\nu_{\mathbf{p}} = \sum_{i\in\mathcal{I}} p(i)\, \nu_{\mathbf{p}}\circ f_i^{-1}.
\end{equation*}
The measure $\nu_{\mathbf{p}} $  has an equivalent characterisation as the push-forward of the Bernoulli measure generated by $\mathbf{p}$  under the natural projection from the symbolic space to the attractor. More precisely, given $\mathbf{p}$, the Bernoulli measure on the symbolic space $\Sigma=\mathcal{I}^{\mathbb{N}}$ is the product measure $\mu_{\mathbf{p}}=\mathbf{p}^{\mathbb{N}}$. The natural projection $\pi:\,\Sigma\to F$ is given by
\begin{equation}\label{eq:14}
	\pi(\ii) =\pi(i_1i_2\ldots i_k \ldots) \coloneqq \lim_{k\to\infty} f_{i_1i_2\ldots i_k}(0),
\end{equation} 
where $f_{i_1i_2\ldots i_k}=f_{i_1}\circ f_{i_2} \circ\ldots\circ f_{i_k}$. Then $\nu_{\mathbf{p}} = \mu_{\mathbf{p}}\circ \pi^{-1}$.

Computing (or estimating) the dimensions of self-affine measures in general is a hard problem.  Moreover, many self-affine measures fail to be doubling (and so have infinite Assouad dimension) and so some conditions are needed in order to obtain sensible results. The specific self-affine measures we are able to handle are those supported on `Bara\'nski type sponges'.  That is, the $A_i$ are diagonal matrices and we assume a  separation condition (the `very strong SPPC', see Definition \ref{def:SPPC}) which, roughly speaking, says that all relevant projections of the measure satisfy the, more familiar, strong separation condition (SSC). For such measures we derive upper and lower bounds for the Assouad and lower dimensions, see Theorem \ref{thm:dimMeasure}. Moreover, the upper and lower bounds agree when $d=2,3$ (see Lemma~\ref{lem:B=Aford23}) and also in many other cases in higher dimensions. It remains an interesting open problem whether our bounds are sharp in full generality, see Question~\ref{ques:1}.   One of the main technical challenges in considering `Bara\'nski type sponges' instead of, for example, those of `Bedford--McMullen' or `Lalley--Gatzouras' type is that we have to control the ratio of the measure of approximate cubes with `different orderings'. As such we develop a number of technical tools which may have further application, e.g. the subdivision argument used in proving Proposition \ref{prop:ratioOfMeasures}.  An interesting consequence of our results is that there can be  a `dimension gap' for such self-affine constructions, even in the plane, see Corollary \ref{cor:dimgap} and Proposition \ref{prop:ex1}.

\section{Main results:  dimension bounds and dimension gaps}\label{sec:spongesresult}

\subsection{Our model and assumptions}\label{sec:SPPC}

We call a self-affine set $F$ a \emph{(self-affine) sponge} if the linear part $A_i$ of each $f_i$ is a diagonal matrix with entries $\big(\lambda_i^{(1)},\ldots,\lambda_i^{(d)}\big)$. When $d=2$ sponges are more commonly referred to as  \emph{self-affine carpets} and when $d=1$ they are self-similar sets. The original model for a self-affine carpet was  introduced independently by Bedford~\cite{Bedford84_phd} and  McMullen~\cite{mcmullen84} and  later generalised by Lalley and Gatzouras~\cite{GatzourasLalley92}, Bara\'nski~\cite{BARANSKIcarpet_2007} and many others.  The dimension theory of self-affine carpets is well-developed, although several interesting questions remain such as the question of whether  self-affine carpets necessarily support  an invariant measure of maximal Hausdorff dimension, see \cite{peressolomyaksurvey}.  A recent breakthrough established that this was false for sponges with $d=3$~\cite{das2017hausdorff}, that is,  the existence of a `dimension gap' was established for certain examples.  This dimension gap result resolved a long standing open problem in dynamical systems.

Generally, much less is known about sponges in dimensions $d\geq 3$. The objective of this paper is to contribute to this line of research. A number of results concern the higher dimensional Bedford--McMullen sponges, see Example~\ref{ex:BM} for the formal definition. Their Hausdorff and box dimensions were determined by Kenyon and Peres~\cite{KenyonPeres_ETDS96} while their Assouad and lower dimensions were calculated by Fraser and Howroyd~\cite{FraserHowroyd_AnnAcadSciFennMath17}. Olsen~\cite{Olsen_PJM98} studied multifractal properties of self-affine measures supported by these sponges and  Fraser and Howroyd~\cite{FraserHowroyd_Indiana20} derived a formula for the Assouad dimension of such measures. The lower and Assouad dimensions of Lalley--Gatzouras sponges, see Example~\ref{ex:LG}, are also known~\cite{das_fishman_simmons_urbanski_2019ETDS, Howroyd_JFG19}.

Without loss of generality we assume that $f_i([0,1]^d)\subset [0,1]^d$ and that there is no $i\neq j$ such that $f_i(x)=f_j(x)$ for every $x\in[0,1]^d$. To avoid unwanted complication with notation, we also assume that
\begin{equation*}
\lambda_i^{(n)}\in(0,1)\; \text{ for every } i\in\mathcal{I} \text{ and } 1\leq n\leq d.
\end{equation*}
We make one further simplification by assuming that all pairs of coordinates are \emph{distinguishable}, i.e. 
\begin{equation}\label{eq:00}
\text{for any } m\neq n\in\{1,\ldots,d\} \text{ there exists } i\in\mathcal{I} \text{ such that } \lambda_i^{(n)} \neq \lambda_i^{(m)}.
\end{equation}
Otherwise, the sponge is not `genuinely self-affine' in all coordinates. The case when not all pairs of coordinates are distinguishable can be handled by `gluing' together non-distinguishable  coordinates as was done by Howroyd~\cite{Howroyd_JFG19} but we omit further discussion of such examples.

The orthogonal projections of $F$ onto the principal $n$-dimensional subspaces play a vital role in the arguments. Let $\mathcal{S}_d$ be the symmetric group on the set $\{1,\ldots,d\}$. For a permutation $\sigma=(\sigma_1,\ldots,\sigma_d)\in\mathcal{S}_d$ of the coordinates, let $E_n^{\sigma}$ denote the $n$-dimensional subspace spanned by the coordinate axes indexed by $\sigma_1,\ldots,\sigma_n$.  Notice that $E_n^{\sigma}=E_n^{\omega}$ as long as $\{\sigma_1,\ldots,\sigma_n\}$ and $\{\omega_1,\ldots,\omega_n\}$ are the same sets. The permutation appears in the notation rather than just the set of indices because the ordering of coordinates will play a role in how the subspace is `built up' from its lower dimensional subspaces. Let $\Pi_n^{\sigma}:[0,1]^d\to E_n^\sigma$ be the orthogonal projection onto $E_n^{\sigma}$. For $n=d$, $\Pi_d^{\sigma}$ is simply the identity map. We say that $f_i$ and $f_j$ \emph{overlap exactly} on $E_n^{\sigma}$ if
\begin{equation*}
	\Pi_n^{\sigma}(f_i(x))=\Pi_n^{\sigma}(f_j(x)) \;\text{ for every }\; x\in[0,1]^d.
\end{equation*}
Observe that if $f_i$ and $f_j$ overlap exactly on $E_n^{\sigma}$ then they also overlap exactly on $E_m^{\sigma}$ for all $1\leq m \leq n$ but may not overlap exactly on any $E_n^{\sigma'}$ for some other $\sigma'\in\mathcal{S}_d$.

Recall $\Sigma= \mathcal{I}^{\mathbb{N}}$ is the space of all one-sided infinite words $\ii=i_1,i_2,\ldots$. Slightly abusing notation, we also write $\ii=i_1,\ldots,i_k\in\mathcal{I}^k$ for a finite length word or $\ii|k=i_1,\ldots,i_k$ for the truncation of $\ii\in\Sigma$. For $r>0$, the \emph{$r$-stopping of $\ii\in\Sigma$ in the $n$-th coordinate} (for $n=1,\ldots,d$) is the unique integer $L_{\ii}(r,n)$ for which
\begin{equation}\label{eq:21}
	\prod_{\ell=1}^{L_{\ii}(r,n)} \lambda_{i_\ell}^{(n)} \leq r < \prod_{\ell=1}^{L_{\ii}(r,n)-1}\lambda_{i_\ell}^{(n)}.
\end{equation}

We distinguish between two different kinds of orderings. We say that $\ii\in\Sigma$ determines a \emph{$\sigma$-ordered cylinder at scale $r$} if  $\sigma_d=\sigma_d(\ii,r)$ is the largest index that satisfies
\begin{equation*}
L_{\ii}(r,\sigma_d)=\min_{n\in\{1,\ldots,d\}} L_{\ii}(r,n) \;\text{ and }\;  \prod_{\ell=1}^{L_{\ii}(r,\sigma_d)} \lambda_{i_{\ell}}^{(\sigma_d)} = \min_{n\in\{1,\ldots,d\}} \prod_{\ell=1}^{L_{\ii}(r,\sigma_d)} \lambda_{i_{\ell}}^{(n)},
\end{equation*}
and then
\begin{equation}\label{eq:201}
	\prod_{\ell=1}^{L_{\ii}(r,\sigma_d)} \lambda_{i_{\ell}}^{(\sigma_d)} \leq \prod_{\ell=1}^{L_{\ii}(r,\sigma_d)} \lambda_{i_{\ell}}^{(\sigma_{d-1})} \leq \ldots \leq \prod_{\ell=1}^{L_{\ii}(r,\sigma_d)} \lambda_{i_{\ell}}^{(\sigma_1)},
\end{equation}
where to make the ordering unique, we use the convention that
\begin{equation*}
\text{ if }\; \prod_{\ell=1}^{L_{\ii}(r,\sigma_d)} \lambda_{i_{\ell}}^{(\sigma_n)} = \prod_{\ell=1}^{L_{\ii}(r,\sigma_{d})} \lambda_{i_{\ell}}^{(\sigma_{n-1})} \;\text{ then } \sigma_n>\sigma_{n-1}.
\end{equation*}
It is a \emph{strictly} $\sigma$-ordered cylinder if all inequalities in~\eqref{eq:201} are strict.  This corresponds to the ordering of the length of the sides of the cylinder set $f_{\ii|L_{\ii}(r,\sigma_d)}([0,1]^d)$  with $\sigma_d$ corresponding to the shortest side and $\sigma_1$ the longest.  Moreover, we say that $\ii\in\Sigma$ determines a \emph{$\sigma$-ordered cube at scale $r$} if 
\begin{equation}\label{eq:202}
	L_{\ii}(r,\sigma_d)\leq L_{\ii}(r,\sigma_{d-1})\leq \ldots\leq L_{\ii}(r,\sigma_1).
\end{equation}
Here the ordering is made unique with the following rule: if coordinates $k<m$ satisfy $L_{\ii}(r,k)=L_{\ii}(r,m)$, then $k$ precedes $m$ in $\sigma$ if and only if $\prod_{\ell=1}^{L_{\ii}(r,k)} \lambda_{i_{\ell}}^{(k)} \geq \prod_{\ell=1}^{L_{\ii}(r,k)} \lambda_{i_{\ell}}^{(m)}$. This corresponds to the ordering of the sides of a symbolic approximate cube to be formally introduced in Section~\ref{sec:symbolic}.  Note that the ordering of $\ii$ as a cylinder or as a cube at a scale $r$ need not be the same. Of importance are the different orderings that are `witnessed' by an $\ii\in\Sigma$ at some scale $r$:
\begin{multline}\label{eq:203}
\mathcal{A}\coloneqq \{ \sigma\in\mathcal{S}_d:\, \text{ there exist } \ii\in\Sigma \text{ and } r>0 \text{ such that } \\
\ii \text{ determines a } \sigma\text{-ordered cube at scale } r \}
\end{multline}
and
\begin{multline}\label{eq:204}
\mathcal{B}\coloneqq \{ \sigma\in\mathcal{S}_d:\, \text{ there exist } \ii\in\Sigma \text{ and } r>0 \text{ such that } \\
\ii \text{ determines a strictly } \sigma\text{-ordered cylinder at scale } r \}.
\end{multline}
Clearly $\mathcal{B}\subseteq \mathcal{A}$ because if $\sigma\in\mathcal{B}$ is witnessed by $\jj$ at scale $r$, then by defining $\ii\coloneqq \overline{\jj|L_{\jj}(r,\sigma_d)}$, i.e. repeating the word $\jj|L_{\jj}(r,\sigma_d)$ infinitely often, there is $r'$ small enough such that~\eqref{eq:202} holds. We give a more detailed account of the relationship between $\mathcal{A}$ and $\mathcal{B}$ in Section~\ref{sec:orderings}, where we show that $\mathcal{A} = \mathcal{B}$ for $d=2$ and $3$, however, also present a four dimensional example for which $\mathcal{B}\subset\mathcal{A}$. A simple example to determine $\mathcal{A}$ and $\mathcal{B}$ is when the sponge $F$ satisfies the \emph{coordinate ordering condition}, i.e. there exists a permutation $\sigma\in\mathcal{S}_d$ such that 
\begin{equation}\label{eq:10}
	0<\lambda_i^{(\sigma_d)}\leq\lambda_i^{(\sigma_{d-1})}\leq\ldots\leq\lambda_i^{(\sigma_1)}<1 \;\text{ for every } i\in\mathcal{I}.
\end{equation}
In this case, $L_{r}(\ii,\sigma_d)\leq L_{r}(\ii,\sigma_{d-1})\leq \ldots\leq L_{r}(\ii,\sigma_1)$ for every $\ii\in\Sigma$ and $r>0$, hence, $\mathcal{A}=\mathcal{B}=\{\sigma\}$ and only the projections $\Pi_n^{\sigma}F$ play a role in the study of $F$.

For each permutation $\sigma\in\mathcal{A}$ we define index sets $\mathcal{I}_d^{\sigma}\supseteq \mathcal{I}_{d-1}^{\sigma}\supseteq \ldots \supseteq \mathcal{I}_1^{\sigma}$ with $\mathcal{I}_d^{\sigma}\coloneqq \mathcal{I}$ as follows. Initially set $\mathcal{I}_d^{\sigma}= \mathcal{I}_{d-1}^{\sigma}= \ldots = \mathcal{I}_1^{\sigma}$ and then repeat the following procedure for all pairs $i<j$ ($i,j\in\mathcal{I}$). Starting from $n=d-1$ and decreasing $n$, check whether $f_i$ and $f_j$ overlap exactly on $E_n^{\sigma}$. If they do not overlap exactly for any $n$, then move onto the next pair $(i,j)$, otherwise, take the largest $n'$ for which $f_i$ and $f_j$ overlap exactly and remove $j$ from $\mathcal{I}_{n'}^{\sigma},\mathcal{I}_{n'-1}^{\sigma},\ldots, \mathcal{I}_{1}^{\sigma}$ and then move onto the next pair $(i,j)$. The sets $\mathcal{I}_{d-1}^{\sigma}, \ldots, \mathcal{I}_1^{\sigma}$ are what remain after repeating this procedure for all pairs $i<j$. Further abusing notation, we denote by $\Pi_n^{\sigma}:\, \mathcal{I}\to\mathcal{I}_n^{\sigma}$ the `projection' of $j\in\mathcal{I}$ onto $\mathcal{I}_n^{\sigma}$, i.e.
\begin{equation*}
	\Pi_n^{\sigma}j=i, \quad\text{ if } f_i \text{ and } f_j \text{ overlap exactly on } E_n^{\sigma} \text{ and } i\in\mathcal{I}_n^{\sigma}.
\end{equation*}
Defining $\Sigma_{n}^{\sigma}\coloneqq (\mathcal{I}_{n}^{\sigma})^{\mathbb{N}}$, we also let $\Pi_n^{\sigma}: \Sigma\to \Sigma_{n}^{\sigma}$ by acting coordinate wise, i.e. $\Pi_n^{\sigma}\ii=\Pi_n^{\sigma} i_1,\Pi_n^{\sigma}i_2,\ldots$. For completeness, let $\Pi_d^{\sigma}$ be the identity map on $\Sigma$.

\begin{defn}\label{def:SPPC}
A self-affine sponge $F\subset [0,1]^d$ satisfies the \emph{separation of principal projections condition} (SPPC) if for every $\sigma\in\mathcal{A}$, $1\leq n\leq d$ and $i,j\in\mathcal{I}$,
\begin{equation}\label{eq:200}
	\text{either } f_i \text{ and } f_j \text{ overlap exactly on } E_n^{\sigma} \text{ or } \Pi_n^{\sigma}\big(f_i((0,1)^d)\big)\cap \Pi_n^{\sigma}\big(f_j((0,1)^d)\big) = \emptyset.
\end{equation}
The sponge satisfies the \emph{very strong SPPC} if $(0,1)^d$ can be replaced with $[0,1]^d$.
\end{defn}

If~\eqref{eq:200} is only assumed for $n=d$, the rather weaker condition is known as  the \emph{rectangular open set condition}, e.g.~\cite{FengWang2005}. The following are the natural generalisations of Bara\'nski~\cite{BARANSKIcarpet_2007}, Lalley--Gatzouras~\cite{GatzourasLalley92} and Bedford--McMullen~\cite{Bedford84_phd, mcmullen84} carpets to higher dimensions.
\begin{example}\label{ex:Baranski}
	A \emph{Bara\'nski sponge} $F\subset[0,1]^d$ satisfies that for all $\sigma\in\mathcal{S}_d$ and $i,j\in\mathcal{I}$,
	\begin{equation*}
		\text{either } f_i \text{ and } f_j \text {overlap exactly on } E_1^{\sigma} \text{ or } \Pi_1^{\sigma}\big(f_i((0,1)^d)\big)\cap \Pi_1^{\sigma}\big(f_j((0,1)^d)\big) = \emptyset.
	\end{equation*}
	In other words, the IFSs generated on the coordinate axes by indices $\mathcal{I}_1^{\sigma}$ satisfy the open set condition. This clearly implies the SPPC. 
\end{example}
\begin{example}\label{ex:LG}
	A \emph{Lalley--Gatzouras sponge} $F\subset[0,1]^d$ satisfies the SPPC and the coordinate ordering condition~\eqref{eq:10} for some $\sigma\in\mathcal{S}_d$.
\end{example}
\begin{example}\label{ex:BM}
	A \emph{Bedford--McMullen sponge} $F\subset[0,1]^d$ is a Bara\'nski sponge which satisfies the coordinate ordering condition (hence, is also a Lalley--Gatzouras sponge) and
	\begin{equation*}
		\lambda_1^{(n)} = \lambda_2^{(n)} =\ldots = \lambda_N^{(n)} \;\text{ for all } 1\leq n\leq d.
	\end{equation*}
\end{example}

Observe that a carpet on the plane satisfies the SPPC if and only if it is either Bara\'nski (when $\#\mathcal{A}=2$) or Lalley--Gatzouras (when $\#\mathcal{A}=1$). Therefore, this definition combines these two classes in a natural way. Moreover, for dimensions $d\geq 3$ it is a wider class of sponges than simply the union of the Bara\'nski and Lalley--Gatzouras class. For $d=3$, we give a complete characterisation of the new classes that emerge in Section~\ref{sec:SPPCdim3}.
 
The very strong SPPC is a natural extension of the \emph{very strong separation condition} first introduced by King~\cite{King_LocalDimBMCarpet_95AdvMath} to study the fine multifractal spectrum of self-affine measures on Bedford--McMullen carpets. It was later adapted to higher dimensional Bedford--McMullen sponges by Olsen~\cite{Olsen_PJM98}. It is also assumed by Fraser and Howroyd~\cite{FraserHowroyd_AnnAcadSciFennMath17, FraserHowroyd_Indiana20} when calculating the Assouad dimension of self-afffine measures on these sponges. In fact, in this case the very strong separation condition is a necessary assumption. Without it, one can construct a carpet which does not carry any doubling self-affine measure, see~\cite[Section~4.2]{FraserHowroyd_AnnAcadSciFennMath17} for an example.

\subsection{Main result}\label{sec:mainresult}

In order to state our main result we need to introduce additional probability vectors derived from $\mathbf{p}=(p(i))_{i\in\mathcal{I}}$ by `projecting' it onto subsets $\mathcal{I}_n^{\sigma}\subseteq\mathcal{I}$. For $\sigma\in\mathcal{A}$ and $1\leq n\leq d-1$ let  
\begin{equation*}
	\mathbf{p}_n^{\sigma}\coloneqq \big( p_n^{\sigma}(i) \big)_{i\in\mathcal{I}_n^{\sigma}}, \;\text{ where } p_n^{\sigma}(i)\coloneqq \sum_{j\in\mathcal{I}:\, \Pi_n^{\sigma}j=i} p(j).
\end{equation*}
Observe that due to the SPPC, $p_n^{\sigma}(i)$ can also be calculated by
\begin{equation}\label{eq:23}
	p_n^{\sigma}(i)= \sum_{j\in\mathcal{I}_{n+1}^{\sigma,i}} p_{n+1}^{\sigma}(j), \;\text{ where } \mathcal{I}_{n+1}^{\sigma,i}\coloneqq \{ j\in\mathcal{I}_{n+1}^{\sigma}:\, \Pi_n^{\sigma}j=i \}.
\end{equation}
This gives rise to the conditional measure $\mathbf{P}_{n-1}^{\sigma,i}=\big( P_{n-1}^{\sigma,i}(j)\big)_{j\in\mathcal{I}_{n}^{\sigma,i}}$ along the fibre $i\in\mathcal{I}_{n-1}^{\sigma}$ for $1\leq n\leq d$ by setting
\begin{equation*}
	P_{n-1}^{\sigma,i}(j) \coloneqq \frac{p_n^{\sigma}(j)}{p_{n-1}^{\sigma}(i)},
\end{equation*}
where in case $n=1$, we define $\Pi_0^{\sigma}i=\emptyset, \mathcal{I}_0^{\sigma}=\{\emptyset\}$ and $p_0^{\sigma}(\emptyset)=1$. This is a natural extension of the conditional probabilities introduced by Olsen~\cite{Olsen_PJM98} for Bedford--McMullen sponges. For $m\geq n$ and $i\in\mathcal{I}_{m}^{\sigma}$, we slightly simplify notation by writing
\begin{equation}\label{eq:11}
	P_{n-1}^{\sigma}(\Pi_{n}^{\sigma}i) = P_{n-1}^{\sigma,\Pi_{n-1}^{\sigma}i}(\Pi_{n}^{\sigma}i).
\end{equation}

A specific choice of $\mathbf{p}$ has particular importance. For $i\in\mathcal{I}_n^{\sigma}$ ($0\leq n\leq d-1$), define $s_n^{\sigma}(i)$ to be the unique number which satisfies the equation
\begin{equation*}
	\sum_{j\in\mathcal{I}_{n+1}^{\sigma,i}} \big( \lambda_j^{(\sigma_{n+1})} \big)^{s_n^{\sigma}(i)}=1.
\end{equation*}
This is the similarity dimension of the IFS given by the ``fibre above" $i$.  The SPPC implies that $s_n^{\sigma}(i)\in[0,1]$. We define the \emph{$\sigma$-ordered coordinate-wise natural measure} as
\begin{equation}\label{eq:12}
	\mathbf{q}^{\sigma} = (q^{\sigma}(i))_{i\in\mathcal{I}}, \;\text{ where } q^{\sigma}(i)\coloneqq \prod_{n=1}^d \big( \lambda_{\Pi_n^{\sigma} i}^{(\sigma_n)} \big)^{s_{n-1}^{\sigma}(\Pi_{n-1}^{\sigma} i)}.
\end{equation}
For Bedford--McMullen sponges, Fraser and Howroyd~\cite{FraserHowroyd_AnnAcadSciFennMath17} used the terminology coordinate \emph{uniform} measure since in that case the natural measure along a fibre simplifies to the uniform measure. This measure has the special property that
\begin{equation*}
q_n^{\sigma}(i) = \sum_{j\in\mathcal{I}:\, \Pi_n^{\sigma}j=i} q^{\sigma}(j) = \prod_{m=1}^n \big( \lambda_{\Pi_m^{\sigma} i}^{(\sigma_m)} \big)^{s_{m-1}^{\sigma}(\Pi_{m-1}^{\sigma} i)}.
\end{equation*}
We are now ready to state our main result.

\begin{thm}\label{thm:dimMeasure}
Let $\nu_{\mathbf{p}}$ be a self-affine measure fully supported on a self-affine sponge satisfying the very strong SPPC. Then
\begin{equation*}
	\max_{\sigma\in\mathcal{B}}\, \overline{S}(\mathbf{p},\sigma) \leq \dim_{\mathrm{A}} \nu_{\mathbf{p}} \leq \max_{\sigma\in\mathcal{A}}\, \overline{S}(\mathbf{p},\sigma) 
\end{equation*}
and
\begin{equation*}
	\min_{\sigma\in\mathcal{A}}\, \underline{S}(\mathbf{p},\sigma) \leq \dim_{\mathrm{L}} \nu_{\mathbf{p}} \leq \min_{\sigma\in\mathcal{B}}\, \underline{S}(\mathbf{p},\sigma) ,
\end{equation*}
where
\begin{equation*}
\overline{S}(\mathbf{p},\sigma) \coloneqq  \sum_{n=1}^{d}\, \max_{i\in\mathcal{I}_n^{\sigma}} \frac{ \log P_{n-1}^{\sigma}(i) }{ \log \lambda_{i}^{(\sigma_n)} } \;\;\text{ and }\;\; \underline{S}(\mathbf{p},\sigma) \coloneqq \sum_{n=1}^{d}\, \min_{i\in\mathcal{I}_n^{\sigma}} \frac{ \log P_{n-1}^{\sigma}(i) }{ \log \lambda_{i}^{(\sigma_n)} }.
\end{equation*}
In particular, for the $\sigma$-ordered coordinate-wise natural measure
\begin{equation*}
	\overline{S}(\mathbf{q}^{\sigma} ,\sigma) = s_0^{\sigma}(\emptyset) + \sum_{n=1}^{d-1}\, \max_{i\in\mathcal{I}_n^{\sigma}} s_n^{\sigma}(i) \;\;\text{ and }\;\; 
	\underline{S}(\mathbf{q}^{\sigma} ,\sigma)  = s_0^{\sigma}(\emptyset) + \sum_{n=1}^{d-1}\, \min_{i\in\mathcal{I}_n^{\sigma}} s_n^{\sigma}(i).
\end{equation*}
\end{thm}

Symbolic arguments used in our proof are collected in Section~\ref{sec:symbolic} while the theorem itself is proved in Section~\ref{sec:proofofmain}. The result generalises the formula in~\cite[Theorem~2.6]{FraserHowroyd_Indiana20} for $\dim_{\mathrm{A}} \nu_{\mathbf{p}}$ in case of Bedford--McMullen sponges. A sufficient condition for the lower and upper bounds to coincide is if $\mathcal{A}=\mathcal{B}$. This occurs when $F$ is a Lalley--Gatzouras sponge in any dimensions, moreover, we prove in Section~\ref{sec:orderings} that $\mathcal{A}=\mathcal{B}$ for all $F$ satisfying the SPPC in dimensions $d=2$ and $3$. However, $\mathcal{A}=\mathcal{B}$ is not a necessary condition. We give an example in four dimensions for which the lower and upper bounds coincide even though $\mathcal{B}\subset \mathcal{A}$, see Proposition~\ref{prop:exd4}. Finding a potential example for $\max_{\sigma\in\mathcal{B}}\, \overline{S}(\mathbf{p},\sigma) < \max_{\sigma\in\mathcal{A}}\, \overline{S}(\mathbf{p},\sigma)$ seems to be a more delicate matter and is a natural direction for further research.

\begin{ques}\label{ques:1}
Is it true that $\max_{\sigma\in\mathcal{B}}\, \overline{S}(\mathbf{p},\sigma) = \max_{\sigma\in\mathcal{A}}\, \overline{S}(\mathbf{p},\sigma)$ even if $\mathcal{B}\subset \mathcal{A}$? If not then what is the correct value of $\dim_{\mathrm{A}}\nu_{\mathbf{p}}$?
\end{ques}

\subsection{A dimension gap: examples and non-examples}\label{sec:dimgap}

Very often it is the case that one of the bounds to determine some dimension of a set is obtained by calculating the respective dimension of measures supported by the set. For example, for the Assouad dimension Luukkainen and Saksman~\cite{LuukkainenSaksman_ProcAMS98} and for the lower dimension Bylund and Gudayol~\cite{BylundGudayol_ProcAMS00} proved that if $F \subseteq \mathbb{R}^d$ is closed, then
\begin{equation*}
	\dim_{\mathrm{A}} F=\inf \left\{\dim_{\mathrm{A}} \nu: \mathrm{supp}(\nu)=F\right\} 
\end{equation*}
and
\begin{equation*}
	\dim_{\mathrm{L}} F=\sup \left\{\dim_{\mathrm{L}} \nu: \mathrm{supp}(\nu)=F\right\}.
\end{equation*}
The well-known mass distribution principle and Frostman's lemma combine to provide a similar result for the Hausdorff dimension, for example, see \cite{FalconerBook}. There is also a relatively new notion of box or `Minkowski' dimension for measures and again there is a similar result, see \cite[Theorem 2.1]{FFK_MinkowskiDimMeas_2020arxiv}.  Therefore, it is interesting to see whether the dimension of a set is still attained by restricting to a certain class of measures (e.g. dynamically invariant measures) or if  there is a strictly positive `dimension gap'.

Self-affine measures supported on carpets and sponges have been used to showcase both kinds of behaviour. Here we just give a few highlights and direct the interested reader to the book~\cite[Chapter 8.5]{fraser_2020Book} for a more in-depth discussion. The Hausdorff dimension of a Lalley--Gatzouras carpet is attained by a self-affine measure~\cite{GatzourasLalley92}, however, this is not the case in higher dimensions by the counterexample of Das and Simmons~\cite{das2017hausdorff}. The box dimension of a Bedford--McMullen carpet is attained by a self-affine measure if and only if the carpet has uniform fibres, see \cite{BJK_Chaos_IMRN22}. The Assouad and lower dimensions of a Lalley--Gatzouras sponge are simultaneously realised by the same self-affine measure, namely the coordinate-wise natural measure~\cite{Howroyd_JFG19}. 

Going beyond the Lalley--Gatzouras class, one might expect that if $\mathcal{A}=\mathcal{B}$ then one of the coordinate-wise natural measures could still realise the Assouad dimension and potentially another the lower dimension. An interesting corollary of Theorem~\ref{thm:dimMeasure} is that this is not the case in general. A strictly positive dimension gap can occur on the plane, noting that $\dim_{\mathrm{A}}F$ and $\dim_{\mathrm{L}}F$ were calculated by Fraser~\cite{Fraser_TAMS2014} using covering arguments.

\begin{cor}\label{cor:dimgap}
There exists a Bara\'nski carpet $F$ such that 
\begin{equation}\label{eq:28}
	\inf_{\mathbf{p}} \dim_{\mathrm{A}} \nu_{\mathbf{p}} \geq  \dim_{\mathrm{A}} F+\delta_F
\end{equation}
for some $\delta_F>0$ depending only on $F$. Moreover, there also exists a Bara\'nski carpet $E$ such that $\dim_{\mathrm{A}} E= \dim_{\mathrm{A}} \nu_{\mathbf{q}^{(1,2)}}$.
\end{cor}
These families of examples are presented in Section~\ref{sec:Bexamples}. Finding conditions under which there is a dimension gap also seems a delicate issue.

\begin{ques}
Is it possible to give simple necessary and/or sufficient conditions for general self-affine  carpets satisfying the very strong SPPC for there to be a dimension gap in the sense of~\eqref{eq:28}?
\end{ques}

An unfortunate consequence of Corollary~\ref{cor:dimgap} is that in general the class of self-affine measures is insufficient to use in order to determine $\dim_{\mathrm{A}}F$.

\begin{ques}
What class $\mathcal{P}$ of measures should be used on the plane to ensure $\inf_{\nu\in\mathcal{P}} \dim_{\mathrm{A}} \nu =  \dim_{\mathrm{A}} F$? For example, can $\mathcal{P}$ be taken to be the set of invariant measures? 
\end{ques}

\section{Comparing orderings of cubes and cylinders}\label{sec:orderings}

In this section we establish some further relationships between $\mathcal{A}$ and $\mathcal{B}$, recall~\eqref{eq:203} and~\eqref{eq:204}. We say that \emph{coordinate $x$ dominates coordinate $y$}, denoted $y\prec x$, if 
\begin{equation}\label{eq:300}
\lambda_i^{(y)} \leq \lambda_i^{(x)} \text{ for every } i\in\mathcal{I}.
\end{equation}
Since any two coordinates $x\neq y$ are distinguishable~\eqref{eq:00}, there actually exists an $i$ for which the inequality is strict. A consequence of \eqref{eq:300} is that $L_{\ii}(r,y)\leq L_{\ii}(r,x)$ for all $\ii\in\Sigma$ and $r>0$, therefore, $x$ must precede $y$ in any $\sigma\in\mathcal{A}$. As a result, if there is a chain of coordinates $x_n\prec x_{n-1}\prec \ldots \prec x_1$, then $\#\mathcal{A}\leq d!/n!$. Moreover, if $y\prec x$, then the orthogonal projection onto the $xy$-plane must be a Lalley--Gatzouras carpet with coordinate $x$ the dominant, while if neither dominates the other, then the projection is a Bara\'nski carpet. In general, we say $F$ is a \emph{genuine} Bara\'nski sponge if there do not exist coordinates $x,y$ with $x \prec y$. An example with two maps is if $\lambda_1^{(d)}<\lambda_1^{(d-1)}<\ldots <\lambda_1^{(1)}$ and $\lambda_2^{(1)}<\lambda_2^{(2)}<\ldots <\lambda_2^{(d)}$.

We start with a useful equivalent characterisation of $\mathcal{B}$ by a condition on the maps of the IFS. Let $\mathcal{P}_{\mathcal{I}}$ denote the set of all probability vectors on $\mathcal{I}$. For a coordinate $x$ and $\mathbf{p}\in\mathcal{P}_{\mathcal{I}}$, we define the \emph{Lyapunov exponent} to be $\chi_x(\mathbf{p}) \coloneqq -\sum_{i\in\mathcal{I}}p(i)\log \lambda_i^{(x)}$. Observe that if $y\prec x$, then $\chi_{x}(\mathbf{p}) < \chi_{y}(\mathbf{p})$ for every $\mathbf{p}\in\mathcal{P}_{\mathcal{I}}$. The following lemma shows that to determine $\mathcal{B}$ it is enough to see how $\mathcal{P}_\mathcal{I}$ gets partitioned by the different orderings of Lyapunov exponents.
\begin{lma}\label{lem:EquivCharacB}
An ordering $\sigma\in\mathcal{B}$ if and only if there exists $\mathbf{p}\in\mathcal{P}_{\mathcal{I}}$ such that $\chi_{\sigma_1}(\mathbf{p}) < \chi_{\sigma_2}(\mathbf{p}) < \ldots < \chi_{\sigma_d}(\mathbf{p})$.
\end{lma}
\begin{proof}
By introducing the empirical probability vector $\mathbf{t}_{\ii}^K=(t_{\ii}^K(i))_{i\in\mathcal{I}}$ with coordinate
\begin{equation*}
	t_{\ii}^K(i)\coloneqq \frac{1}{K} \#\{1\leq k\leq K:\, i_{k}=i\}
\end{equation*}
for $\ii\in\Sigma$, $K\in\mathbb{N}$ and $i\in\mathcal{I}$, we can express for any coordinate $n$,
\begin{equation*}
	\prod_{\ell=1}^{K} \lambda_{i_{\ell}}^{(n)} = \prod_{i\in\mathcal{I}} \big( \lambda_{i}^{(n)} \big)^{K\cdot t_{\ii}^K(i)} = \exp\Big[ K\sum_{i\in\mathcal{I}}t_{\ii}^K(i)\log \lambda_{i_{\ell}}^{(n)}  \Big] = \exp \big[ -K\cdot \chi_n(\mathbf{t}_{\ii}^K) \big].
\end{equation*}
By definition, if $\sigma\in\mathcal{B}$ then there exist $\ii\in\Sigma$ and $r>0$ such that
\begin{equation*}
	\prod_{\ell=1}^{L_{\ii}(r,\sigma_d)} \lambda_{i_{\ell}}^{(\sigma_d)} < \prod_{\ell=1}^{L_{\ii}(r,\sigma_d)} \lambda_{i_{\ell}}^{(\sigma_{d-1})} <\ldots < \prod_{\ell=1}^{L_{\ii}(r,\sigma_d)} \lambda_{i_{\ell}}^{(\sigma_1)}.
\end{equation*}
This clearly implies $\chi_{\sigma_1}(\mathbf{p}) < \chi_{\sigma_2}(\mathbf{p}) < \ldots < \chi_{\sigma_d}(\mathbf{p})$ with $\mathbf{p}=\mathbf{t}_{\ii}^{L_{\ii}(r,\sigma_d)}$.
	
Conversely, if $\chi_{\sigma_1}(\mathbf{p}) < \chi_{\sigma_2}(\mathbf{p}) < \ldots < \chi_{\sigma_d}(\mathbf{p})$ then there also exists $\mathbf{q}\in\mathcal{P}_{\mathcal{I}}$ arbitrarily close to $\mathbf{p}$ with the property that each element has the form $q(i)=a_i/K$ for some $a_i,K\in\mathbb{N}$ and still $\chi_{\sigma_1}(\mathbf{q}) < \chi_{\sigma_2}(\mathbf{q}) < \ldots < \chi_{\sigma_d}(\mathbf{q})$. Then any $\ii\in\Sigma$ such that $\mathbf{t}_{\ii}^K=\mathbf{q}$ and $r=\prod_{\ell=1}^{K} \lambda_{i_{\ell}}^{(\sigma_d)}$ shows that $\sigma\in\mathcal{B}$.
\end{proof}

Consider the set $\mathcal{Q}\coloneqq\{\mathbf{p}\in\mathcal{P}_{\mathcal{I}}:\, \text{ there exist } x\neq y \text{ such that } \chi_{x}(\mathbf{p}) = \chi_{y}(\mathbf{p})\}$. It is the union of lower dimensional slices of $\mathcal{P}_{\mathcal{I}}$. Since all pairs of coordinates are distinguishable~\eqref{eq:00}, for every $\mathbf{q}\in\mathcal{Q}$ with $\chi_{\sigma_1}(\mathbf{q}) \leq \chi_{\sigma_2}(\mathbf{q}) \leq \ldots \leq \chi_{\sigma_d}(\mathbf{q})$ there exists $\mathbf{p}\in\mathcal{P}_{\mathcal{I}}\setminus \mathcal{Q}$ with $\chi_{\sigma_1}(\mathbf{p}) < \chi_{\sigma_2}(\mathbf{p}) < \ldots < \chi_{\sigma_d}(\mathbf{p})$. Therefore, dropping the word `strictly' from the definition of $\mathcal{B}$ in~\eqref{eq:204} gives the same set of orderings.

The relationship $\mathcal{B}\subseteq\mathcal{A}$ always holds. It is interesting to see whether the inclusion is strict or not.
\begin{lma}\label{lem:B=Aford23}
If $d=2$ or $3$, then $\mathcal{A}=\mathcal{B}$ for every sponge $F$ satisfying the SPPC.
\end{lma}

\begin{proof}
For $d=2$ the claim is automatic. For $d=3$, choose $\sigma\in\mathcal{A}$. Then there exists $\ii\in\Sigma$ and $r>0$ such that $L_{\ii}(r,\sigma_3)\leq L_{\ii}(r,\sigma_{2})\leq L_{\ii}(r,\sigma_1)$. We claim that the cylinder $f_{\ii|L_{\ii}(r,\sigma_{2})}([0,1]^d)$ is $\sigma$-ordered. Indeed, the way we have made the $\sigma$-ordering unique implies that 
\begin{equation*}
\prod_{\ell=1}^{L_{\ii}(r,\sigma_2)} \lambda_{i_{\ell}}^{(\sigma_3)} \leq \prod_{\ell=1}^{L_{\ii}(r,\sigma_2)} \lambda_{i_{\ell}}^{(\sigma_2)} \leq \prod_{\ell=1}^{L_{\ii}(r,\sigma_2)} \lambda_{i_{\ell}}^{(\sigma_1)},
\end{equation*}
which is equivalent to  $\chi_{\sigma_1}(\mathbf{p}) \leq \chi_{\sigma_2}(\mathbf{p}) \leq \chi_{\sigma_3}(\mathbf{p})$ with $\mathbf{p}=\mathbf{t}_{\ii}^{L_{\ii}(r,\sigma_2)}$. If the cylinder is not strictly $\sigma$-ordered, then based on the discussion before Lemma~\ref{lem:B=Aford23} one can construct a strictly $\sigma$-ordered cylinder from a small perturbation of $\mathbf{p}$.
\end{proof}

However, in four dimensions the inclusion $\mathcal{B}\subseteq\mathcal{A}$ can be strict. Our example relies on the following lemma.
\begin{lma}\label{lem:CondForD4ex}
Assume the sponge $F$ satisfying the SPPC is the attractor of an IFS consisting of two maps $f_1,f_2$ ordered $(2,1,3,4)$ and $(1,2,4,3)$, respectively. Then
\begin{equation}\label{eq:32}
(1,2,3,4)\in\mathcal{B} \;\Longleftrightarrow\; (2,1,4,3)\notin\mathcal{B} \;\Longleftrightarrow\; \frac{\log \big( \lambda_1^{(2)} / \lambda_1^{(1)} \big)}{\log \big( \lambda_1^{(3)} / \lambda_1^{(4)} \big)} < \frac{\log \big( \lambda_2^{(1)} / \lambda_2^{(2)} \big)}{\log \big( \lambda_2^{(4)} / \lambda_2^{(3)} \big)}.
\end{equation}
\end{lma}
\begin{proof}
Projection of $F$ onto the $(1,2)$-plane or the $(3,4)$-plane is a Bara\'nski carpet while coordinates $3$ and $4$ are dominated by coordinates $1$ and $2$. Therefore, $\mathcal{B}\subseteq\{(2,1,3,4),(1,2,4,3),(2,1,4,3),(1,2,3,4)\}$.

Lemma~\ref{lem:EquivCharacB} implies that $(1,2,3,4)\in\mathcal{B}$ if and only if there exists $\mathbf{p}=(p,1-p)$ such that $\chi_{1}(\mathbf{p}) < \chi_{2}(\mathbf{p}) < \chi_{3}(\mathbf{p}) < \chi_{4}(\mathbf{p})$. Notice that $\chi_{2}(\mathbf{p}) < \chi_{3}(\mathbf{p})$ for any $p$ because coordinate $2$ dominates coordinate $3$. From the other two inequalities $\chi_{1}(\mathbf{p}) < \chi_{2}(\mathbf{p})$ and $\chi_{3}(\mathbf{p}) < \chi_{4}(\mathbf{p})$, we can express $p$ to obtain
\begin{equation}\label{eq:33}
\frac{\log \big( \lambda_2^{(4)} / \lambda_2^{(3)} \big)}{\log \big( (\lambda_2^{(4)}\lambda_1^{(3)}) / (\lambda_2^{(3)}\lambda_1^{(4)}) \big)} < p < 
\frac{\log \big( \lambda_2^{(1)} / \lambda_2^{(2)} \big)}{\log \big( (\lambda_2^{(1)}\lambda_1^{(2)}) / (\lambda_2^{(2)}\lambda_1^{(1)}) \big)}.
\end{equation} 
Straightforward algebraic manipulations show that this is a non-empty interval if and only if the condition on the right hand side of~\eqref{eq:32} holds.

Similarly, $(2,1,4,3)\in\mathcal{B}$ if and only if there exists $\mathbf{p}=(p,1-p)$ such that $\chi_{2}(\mathbf{p}) < \chi_{1}(\mathbf{p}) < \chi_{4}(\mathbf{p}) < \chi_{3}(\mathbf{p})$. This gives the same condition for $p$ as in~\eqref{eq:33} with the inequality signs reversed, which is equivalent to the reversed inequality in~\eqref{eq:32}.
\end{proof}
 
\begin{prop}\label{prop:exd4}
There exists a sponge in four dimensions satisfying the very strong SPPC for which $\mathcal{B}\subset \mathcal{A}$.  Nonetheless, $\max_{\sigma\in\mathcal{B}}\, \overline{S}(\mathbf{p},\sigma) = \max_{\sigma\in\mathcal{A}}\, \overline{S}(\mathbf{p},\sigma)$.
\end{prop}
\begin{proof}
The example consists of just two maps. Map $f_1$ is $(2,1,3,4)$-ordered with
\begin{equation*}
\big(\lambda_1^{(1)},\lambda_1^{(2)},\lambda_1^{(3)},\lambda_1^{(4)}\big) = (0.2, 0.4, 0.08, 0.02)  
\end{equation*}
and $f_2$ is $(1,2,4,3)$-ordered with
\begin{equation*}
 \big(\lambda_2^{(1)},\lambda_2^{(2)},\lambda_2^{(3)},\lambda_2^{(4)}\big) = (0.6, 0.3, 0.1, 0.2).
\end{equation*}	
The translations can clearly be chosen so that the very strong SPPC holds, moreover, $\mathcal{A}\subseteq\{(2,1,3,4),(1,2,4,3),(2,1,4,3),(1,2,3,4)\}$ for the same reason as in the proof of Lemma~\ref{lem:CondForD4ex}.

Our first claim is that $(1,2,3,4)\in\mathcal{A}$. Some calculations show that choosing $r=5\times 10^{-5}$ and $\ii=11112222222\ldots$ yields
\begin{equation*}
L_{\ii}(r,4) =3 < L_{\ii}(r,3) =4 < L_{\ii}(r,2) =10 < L_{\ii}(r,1) =11.
\end{equation*}
It is also easy to check that the parameters do not satisfy the condition on the right hand side of~\eqref{eq:32}, hence, $(1,2,3,4)\notin\mathcal{B}$ by Lemma~\ref{lem:CondForD4ex}. A simple application of Theorem~\ref{thm:dimMeasure} shows that  $\max_{\sigma\in\mathcal{B}}\, \overline{S}(\mathbf{p},\sigma) = \max_{\sigma\in\mathcal{A}}\, \overline{S}(\mathbf{p},\sigma)$.
\end{proof}

\section{Examples}\label{sec:examples}

\subsection{Planar Bara\'nski carpets with different behaviour}\label{sec:Bexamples}

The Assouad dimension of planar Bara\'nski carpets $F$ was determined by Fraser~\cite{Fraser_TAMS2014}. Using our Theorem~\ref{thm:dimMeasure}, we can check whether $\dim_{\mathrm{A}}F=\dim_{\mathrm{A}}\nu_{\mathbf{p}}$ for some self-affine measure $\nu_{\mathbf{p}}$ or if there is a dimension gap in the sense of~\eqref{eq:28}. Surprisingly, both behaviours are witnessed by simple families of examples. Recall, in the Lalley--Gatzouras class $\dim_{\mathrm{A}}F$ is always achieved by the (only) coordinate-wise natural measure.

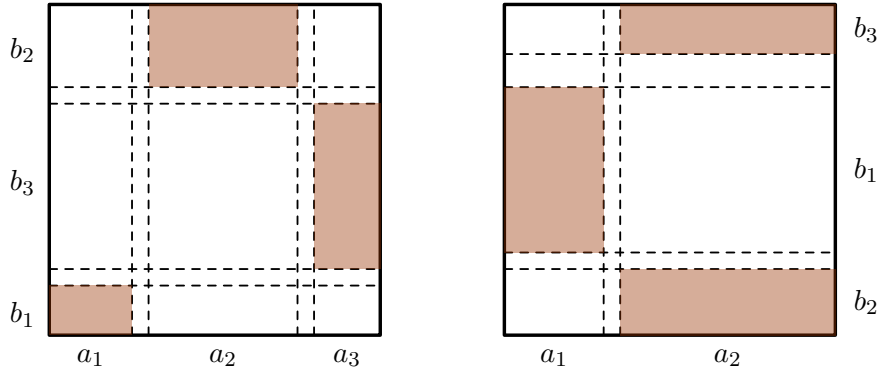
\begin{figure}[ht]
\centering
	
\begin{tikzpicture}[line cap=round,line join=round,x=1cm,y=1cm,scale=1.1]
	\draw[line width=1.3pt] (-5,1) rectangle (-1,5);
	\foreach \y in {1.6,1.8,3.8,4}{
		\draw[dashed, line width=0.7pt] (-5,\y)-- (-1,\y);
	}
	\foreach \x in {-4,-3.8,-2,-1.8}{
		\draw[dashed, line width=0.7pt] (\x,1)-- (\x,5);
	}		
	\draw (-4.5,0.95) node[anchor=north] {$a_1$};
	\draw (-2.9,0.95) node[anchor=north] {$a_2$};
	\draw (-1.4,0.95) node[anchor=north] {$a_3$};
	\draw (-5.05,1.25) node[anchor=east] {$b_1$};
	\draw (-5.05,2.85) node[anchor=east] {$b_3$};
	\draw (-5.05,4.45) node[anchor=east] {$b_2$};
	\fill[line width=1pt,color=zzttqq,fill=zzttqq,fill opacity=0.4] (-5,1) -- (-4,1) -- (-4,1.6) -- (-5,1.6) -- cycle;
	\fill[line width=1pt,color=zzttqq,fill=zzttqq,fill opacity=0.4] (-1.8,1.8) -- (-1,1.8) -- (-1,3.8) -- (-1.8,3.8) -- cycle;
	\fill[line width=1pt,color=zzttqq,fill=zzttqq,fill opacity=0.4] (-3.8,4) -- (-2,4) -- (-2,5) -- (-3.8,5) -- cycle;
		
	\begin{scope}[xshift=-0.5cm]
		\draw[line width=1.3pt] (1,1) rectangle (5,5);	
		\foreach \y in {1.8,2,4,4.4}{
			\draw[dashed, line width=0.7pt] (1,\y)-- (5,\y);
		}
		\foreach \x in {2.2,2.4}{
			\draw[dashed, line width=0.7pt] (\x,1)-- (\x,5);
		}		
		\draw (1.6,0.95) node[anchor=north] {$a_1$};
		\draw (3.7,0.95) node[anchor=north] {$a_2$};
		\draw (5.1,3) node[anchor=west] {$b_1$};
		\draw (5.1,1.4) node[anchor=west] {$b_2$};
		\draw (5.1,4.7) node[anchor=west] {$b_3$};
		\fill[line width=1pt,color=zzttqq,fill=zzttqq,fill opacity=0.4] (2.4,1) -- (5,1) -- (5,1.8) -- (2.4,1.8) -- cycle;
		\fill[line width=1pt,color=zzttqq,fill=zzttqq,fill opacity=0.4] (1,2) -- (2.2,2) -- (2.2,4) -- (1,4) -- cycle;
		\fill[line width=1pt,color=zzttqq,fill=zzttqq,fill opacity=0.4] (2.4,4.4) -- (5,4.4) -- (5,5) -- (2.4,5) -- cycle;
	\end{scope}
\end{tikzpicture}
\caption{Defining maps for a Bara\'nski carpet with strictly positive dimension gap (left), and where the Assouad dimension of $F$ is attained for correctly chosen parameters (right).}
\label{fig:PlanarBaranski}
\end{figure}

Our first example shows a positive dimension gap. Let $F$ be a Bara\'nski carpet which is not in the Lalley--Gatzouras class that satisfies the very strong SPPC with its first level cylinders arranged in a way that there is \emph{no} exact overlap when projecting to either coordinate axis, see left hand side of Figure~\ref{fig:PlanarBaranski} for an example. In particular, this contains all genuine Bara\'nski carpets defined by two maps. Let $a_i=\lambda_i^{(1)}$ and $b_i=\lambda_i^{(2)}$, moreover, define $s$ and $t$ to be the unique solutions to the equations
\begin{equation*}
\sum_{i\in\mathcal{I}} a_i^s=1 \;\;\text{ and }\;\; \sum_{i\in\mathcal{I}} b_i^t=1.
\end{equation*}
Without loss of generality we assume that $t\leq s$. The very strong SPPC implies that $s<1$. The formula from~\cite{Fraser_TAMS2014} shows that $\dim_{\mathrm{L}}F=t\leq s=\dim_{\mathrm{A}}F$.

\begin{prop}\label{prop:ex1}
For a Bara\'nski carpet $F$ described above there is a strictly positive dimension gap, i.e. there exists $\delta_F>0$ such that 
\begin{equation*}
\inf_{\mathbf{p}} \dim_{\mathrm{A}}\nu_{\mathbf{p}} \geq \dim_{\mathrm{A}}F + \delta_F.
\end{equation*}
\end{prop}
\begin{proof}
The condition that there is no exact overlap when projecting to either coordinate axis implies that $p_1^{\sigma}(i)=p(i)$ and so $P_1^{\sigma}(i)=1$ for all $i\in\mathcal{I}=\{1,\ldots,N\}$. Applying Theorem~\ref{thm:dimMeasure}, we immediately obtain
\begin{equation*}
\dim_{\mathrm{A}}\nu_{\mathbf{p}} = \max\Big\{ \frac{\log p(i)}{\log a_i}, \frac{\log p(i)}{\log b_i}:\, i\in\mathcal{I} \Big\}.
\end{equation*}

Since $F$ is not in the Lalley--Gatzouras class and $s\geq t$, there exists $\ell\in\mathcal{I}$ such that $b_{\ell}^s>a_{\ell}^s$. Fix $0<\varepsilon<b_{\ell}^s-a_{\ell}^s$ and first consider any $\mathbf{p}$ that satisfies $p(i)\leq a_i^s+\varepsilon$ for every $i\in\mathcal{I}$. Then
\begin{equation*}
\dim_{\mathrm{A}}\nu_{\mathbf{p}} \geq \frac{\log p(\ell)}{\log b_{\ell}} \geq \frac{\log (a_{\ell}^s+\varepsilon)}{\log b_{\ell}} > s
\end{equation*}
by the choice of $\varepsilon$.

Now assume that $\mathbf{p}$ is such that there exists $j\in\mathcal{I}$ satisfying $p(j)>a_j^s+\varepsilon$. Since $\sum_{i\in\mathcal{I}} a_i^s=1$, the pigeon hole principle implies that there exists $k\in\mathcal{I}$ such that $0<p(k)\leq a_k^s-\varepsilon/(N-1)$. Using this particular index,
\begin{equation*}
\dim_{\mathrm{A}}\nu_{\mathbf{p}} \geq \frac{\log p(k)}{\log a_{k}} \geq \frac{\log\big( a_k^s -\varepsilon/(N-1)\big)}{\log a_k} = s+ \frac{\log\big( (1- \varepsilon\cdot a_k^{-s}/(N-1) )\big)}{\log a_k} > s.
\end{equation*}
Therefore, choosing
\begin{equation*}
\delta_F \coloneqq \min \Big\{ \frac{\log (a_{\ell}^s+\varepsilon)}{\log b_{\ell}} - s, \frac{\log\big( (1-\varepsilon\cdot a_k^{-s}/(N-1))\big)}{\log a_k} \Big\}
\end{equation*}
completes the proof.
\end{proof}

Proposition~\ref{prop:ex1} shows that if a genuine Bara\'nski carpet whose Assouad dimension is realised by a self-affine measure exists, then its defining IFS must have at least three maps. Our second example shows that such a carpet does exist using only three maps. Giving a complete characterisation for Bara\'nski carpets with three maps seems possible but perhaps tedious. However,  it is straightforward to give an easy to check sufficient condition (valid for all Bara\'nski carpets) ensuring that the Assouad dimension of the carpet is attained by a Bernoulli measure.  Comparing the formula from~\cite{Fraser_TAMS2014} with Theorem~\ref{thm:dimMeasure} shows that $\dim_{\mathrm{A}}F = \max_{\sigma\in\mathcal{A}}\overline{S}\big(\mathbf{q}^{\sigma} ,\sigma\big)$. Therefore, if $\sigma$ satisfies that $\overline{S}\big(\mathbf{q}^{\sigma} ,\sigma\big)\geq \max\big\{ \overline{S}\big(\mathbf{q}^{\omega} ,\omega\big),\, \overline{S}\big(\mathbf{q}^{\sigma} ,\omega\big) \big\}$, then $\dim_{\mathrm{A}}F = \dim_{\mathrm{A}} \nu_{\mathbf{q}^{\sigma}}$.

To demonstrate this, consider the Bara\'nski carpet whose first level cylinders are depicted with the three shaded rectangles on the right hand side of Figure~\ref{fig:PlanarBaranski}. To ensure the attractor is a genuine Bara\'nski carpet, assume $a_1<b_1$ and $a_2 > \min\{b_2,b_3\}$.  Define $r,s,t$ as follows: $b_2^r+b_3^r=1,\, a_1^s+a_2^s=1$ and $b_1^t+b_2^t+b_3^t=1$. We assume $\max\{s,t\}<1$ so that the maps can be arranged in a way that satisfies the very strong SPPC. It follows from the formulas in~\cite{Fraser_TAMS2014} that $\dim_{\mathrm{A}}F = \max\{s+r,t\}$.
\begin{prop}\label{prop:ex2}
Consider a Bara\'nski carpet as on the right hand side of Figure~\ref{fig:PlanarBaranski}. Assume $s+r> t$. If $a_2\geq \max\{b_2,b_3\}$ and $b_1^{1+r/s}\leq a_1<b_1$, then $\dim_{\mathrm{A}} F= \dim_{\mathrm{A}} \nu_{\mathbf{q}^{\sigma}}$ for $\sigma=(1,2)$.
\end{prop}

\begin{proof}
Let $\sigma=(1,2)$ and $\omega=(2,1)$. The vector $\mathbf{q}^{\sigma}$ is
\begin{equation*}
	q(1)=a_1^s,\; q(2) = a_2^s\cdot b_2^r \;\;\text{ and }\;\; q(3) = a_2^s\cdot b_3^r.
\end{equation*}
A simple calculation gives $\overline{S}\big(\mathbf{q}^{\sigma} ,\sigma\big)=s+r=\dim_{\mathrm{A}} F$, because we assume $s+r>t$. Hence, it is enough to check when $\overline{S}\big(\mathbf{q}^{\sigma}, \omega\big)\leq s+r$. Another calculation yields
\begin{align*}
	\overline{S}\big(\mathbf{q}^{\sigma}, \omega\big) &= \max_{i\in\mathcal{I}_1^{\omega}} \frac{ \log P_{0}^{\omega}(i) }{ \log \lambda_{i}^{(\omega_1)} } + \max_{i\in\mathcal{I}_2^{\omega}} \frac{ \log P_{1}^{\omega}(i) }{ \log \lambda_{i}^{(\omega_2)} } \\
	&=  \max\Big\{ s\cdot\frac{\log a_1}{\log b_1},\, s\cdot\frac{\log a_2}{\log b_2} +r,\, s \cdot\frac{\log a_2}{\log b_3} +r \Big\} + 0.
\end{align*}
The last two terms are at most $s+r$ if and only if $a_2\geq \max\{b_2,b_3\}$ and  the first term is at most $s+r$ if and only if $b_1^{1+r/s}\leq a_1$, completing the proof.
\end{proof}

\subsection{Characterisation of SPPC in dimension three}\label{sec:SPPCdim3}

Recall from Lemma~\ref{lem:B=Aford23} that $\mathcal{A}=\mathcal{B}$ in $d=3$ and notation $y\prec x$ from~\eqref{eq:300}. If no coordinate dominates any of the other, then (recall) $F$ is a genuine Bara\'nski sponge and projection to any of the three principal planes is a Bara\'nski carpet. For genuine Bara\'nski sponges with $d=3$, it is  not necessarily true  that $\mathcal{B}=\mathcal{S}_3$.  For example, take an IFS consisting of two maps with $\lambda_1^{(3)}<\lambda_1^{(2)}<\lambda_1^{(1)}$ and $\lambda_2^{(1)}<\lambda_2^{(2)}<\lambda_2^{(3)}$. Depending on these parameters, only one of the orderings $(1,3,2)$ and $(2,3,1)$ is an element of $\mathcal{B}$ because there is no $\mathbf{p}\in\mathcal{P}_{\mathcal{I}}$ which simultaneously satisfies $\chi_{1}(\mathbf{p}) < \chi_{3}(\mathbf{p}) < \chi_{2}(\mathbf{p})$ and $\chi_{2}(\mathbf{p}) < \chi_{3}(\mathbf{p}) < \chi_{1}(\mathbf{p})$. With an analogous reasoning, either $(2,1,3)\in\mathcal{B}$ or $(3,1,2)\in\mathcal{B}$. 

Assuming that $y\prec x$, we have $\mathcal{A}\subseteq\{(x,y,z), (x,z,y), (z,x,y)\}$. Hence, projection onto $yz$-plane never plays a role. If $\#\mathcal{A}=1$, then $F$ is a Lalley--Gatzouras sponge. There are potentially three possibilities for $\#\mathcal{A}=2$:
\begin{enumerate}
\item $\mathcal{A}=\{(x,y,z),(x,z,y)\}$, i.e. $\max\{y,z\}\prec x$. In this case, the projection onto both the $xy$ and $xz$-planes are Lalley--Gatzouras carpets with $x$ being the dominant side.
\item $\mathcal{A}=\{(x,z,y),(z,x,y)\}$, i.e. $y\prec\min\{x,z\}$. In this case, projection onto the $xz$-plane is a Bara\'nski carpet and $y\prec\min\{x,z\}$ implies that projection onto either the $xy$-plane or $yz$-plane can be arbitrary.
\end{enumerate}
The third option is not possible due to the following.
\begin{prop}\label{prop:ex3}
Let $F$ be a three dimensional sponge that satisfies the SPPC, $y\prec x$ and $(x,y,z),(z,x,y)\in\mathcal{A}$. Then also $(x,z,y)\in\mathcal{A}$.
\end{prop}

\begin{proof}
All maps of the IFS defining $F$ can not be ordered the same way, therefore, without loss of generality we assume that
\begin{equation}\label{eq:31}
\lambda_1^{(z)} < \lambda_1^{(y)} < \lambda_1^{(x)} \;\;\text{ and }\;\; \lambda_2^{(y)} < \lambda_2^{(x)} < \lambda_2^{(z)}
\end{equation}
corresponding to ordering $(x,y,z)$ and $(z,x,y)$, respectively. We also assume that $f_1$ and $f_2$ do not overlap exactly on neither the $xy$ nor on the $xz$-plane.

According to Lemma~\ref{lem:EquivCharacB} it is enough to show that there exists $\mathbf{p}\in\mathcal{P}_{\mathcal{I}}$ such that $\chi_x(\mathbf{p}) < \chi_{z}(\mathbf{p}) < \chi_{y}(\mathbf{p})$. Consider $\mathbf{p}=(p,1-p,0,\ldots,0)$ noting that the calculations that follow can also be adapted to small enough perturbations of $\mathbf{p}$. Straightforward algebraic manipulations yield
\begin{equation*}
\chi_x(\mathbf{p}) < \chi_{z}(\mathbf{p}) \;\Longleftrightarrow\; p>A \;\;\text{ and }\;\; \chi_{z}(\mathbf{p}) < \chi_{y}(\mathbf{p}) \;\Longleftrightarrow\; p<B,
\end{equation*}
where
\begin{equation*}
A\coloneqq \frac{\log \big(\lambda_2^{(x)}/\lambda_2^{(z)}\big)}{\log \big( (\lambda_2^{(x)}\lambda_1^{(z)}) / (\lambda_2^{(z)}\lambda_1^{(x)}) \big)}  \;\;\text{ and }\;\;
B\coloneqq \frac{\log \big(\lambda_2^{(y)}/\lambda_2^{(z)}\big)}{\log \big( (\lambda_2^{(y)}\lambda_1^{(z)}) / (\lambda_2^{(z)}\lambda_1^{(y)}) \big)}.
\end{equation*}
Additional manipulations show $A<B$ if and only if 
\begin{equation*}
\frac{\log \big(\lambda_1^{(z)}/\lambda_1^{(y)}\big)}{\log \big( \lambda_2^{(y)} / \lambda_2^{(z)} \big)} <
\frac{\log \big(\lambda_1^{(z)}/\lambda_1^{(x)}\big)}{\log \big( \lambda_2^{(x)} / \lambda_2^{(z)} \big)},
\end{equation*}
which is always true because of~\eqref{eq:31}. This completes the proof.
\end{proof}

\section{Symbolic arguments}\label{sec:symbolic}
 
In this section we work on the symbolic space $\Sigma = \mathcal{I}^{\mathbb{N}}$ of all one-sided infinite words $\ii=i_1,i_2,\ldots$ with the Bernoulli measure $\mu_{\mathbf{p}}=\mathbf{p}^{\mathbb{N}}$. Recall all notation from Section~\ref{sec:spongesresult}. Throughout the section a $\sigma$-order always refers to a cube as defined in~\eqref{eq:202}. We define symbolic cubes whose image under the natural projection~\eqref{eq:14} well-approximate Euclidean balls on the sponge $F$. Let $\Sigma_{r}^{\sigma}\coloneqq \{\ii\in\Sigma:\, \ii \text{ is } \sigma\text{-ordered at scale } r\}$. We define the \emph{$\sigma$-ordered symbolic $r$-approximate cube} containing $\ii\in\Sigma_{r}^{\sigma}$ to be
\begin{equation}\label{eq:20}
	B_{\ii}(r)\coloneqq \left\{\jj \in \Sigma: \left|\Pi_{n}^{\sigma} \jj \wedge \Pi_{n}^{\sigma} \ii\right| \geq L_{\ii}(r, \sigma_{n}) \text { for every } 1 \leq n \leq d\right\},
\end{equation}
where $\ii\wedge\jj$ denotes the longest common prefix of $\ii$ and $\jj$. This is the natural extension of the notion of approximate squares used extensively in the study of planar carpets. Due to~\eqref{eq:21}, the image $\pi(B_{\ii}(r))$ is contained within a hypercuboid of $[0,1]^d$ aligned with the coordinate axes with side lengths at most $r$. Observe that if $\ii\in\Sigma_{r}^{\sigma}$, then for all $\jj\in B_{\ii}(r)$ also $\jj\in \Sigma_{r}^{\sigma}$. Thus, we identify the $\sigma$-ordering of $B_{\ii}(r)$ with the $\sigma$-ordering of $\ii$ at scale $r$. If $\ii\in\Sigma_{r}^{\sigma}$, then the surjectivity of the maps $\Pi_n^{\sigma}$ implies that $B_{\ii}(r)$ can be identified with a sequence of symbols of length $L_{\ii}(r,\sigma_1)$ of the form 
\begin{equation*}
	\left(\Pi_{n}^{\sigma} i_{L_{\ii}(r, \sigma_{n+1})+1}, \ldots,\Pi_{n}^{\sigma} i_{L_{\ii}(r, \sigma_{n})}\right)_{n=1}^{d} \in \bigtimes_{n=1}^d (\mathcal{I}_{n}^{\sigma})^{L_{\ii}(r,\sigma_n)-L_{\ii}(r,\sigma_{n+1})},
\end{equation*}
where we set $L_{\ii}(r,\sigma_{d+1})\coloneqq0$.

The following lemmas collect important properties about the $\mu_{\mathbf{p}}$ measure of a symbolic $r$-approximate cube. The first one is the extension of~\cite[eq. (6.2)]{Olsen_PJM98}. We use the convention that any empty product is equal to one. 

\begin{lma}\label{lem:measureApproxCube}
The $\mu_{\mathbf{p}}$ measure of a $\sigma$-ordered symbolic $r$-approximate cube is equal to
\begin{equation*}
\mu_{\mathbf{p}}(B_{\ii}(r)) = \prod_{n=1}^{d}\, \prod_{\ell=L_{\ii}(r,\sigma_{n+1})+1}^{L_{\ii}(r,\sigma_n)} p_n^{\sigma}(\Pi_n^{\sigma}i_{\ell})
= \prod_{n=1}^{d} \prod_{\ell=1}^{L_{\ii}(r,\sigma_n)} P_{n-1}^{\sigma}(\Pi_{n}^{\sigma} i_{\ell}).
\end{equation*}
\end{lma}
\begin{proof}
From definition~\eqref{eq:20} of $B_{\ii}(r)$ it follows that an approximate cube is the disjoint union of level $L_{\ii}(r,\sigma_1)$ cylinder sets:
\begin{equation*}
\big\{ [j_1,\ldots,j_{L_{\ii}(r,\sigma_1)}]:\, \Pi_{n}^{\sigma} j_{\ell}=\Pi_n^{\sigma} i_{\ell} \text{ for } \ell=L_{\ii}(r,\sigma_{n+1})+1,\ldots,L_{\ii}(r,\sigma_n) \text{ and } 1\leq n\leq d \big\}.
\end{equation*}
For each such cylinder, $\mu_{\mathbf{p}}([j_1,\ldots,j_{L_{\ii}(r,\sigma_1)}]) = \prod_{\ell=1}^{L_{\ii}(r,\sigma_1)} p(j_{\ell})$. Adding up and using multiplicativity, we obtain
\begin{equation*}
	\mu_{\mathbf{p}}(B_{\ii}(r)) = \prod_{n=1}^{d}\, \prod_{\ell=L_{\ii}(r,\sigma_{n+1})+1}^{L_{\ii}(r,\sigma_n)}\, \sum_{j\in\mathcal{I}:\, \Pi_n^{\sigma} j=\Pi_n^{\sigma} i_{\ell}} p(j) = \prod_{n=1}^{d}\, \prod_{\ell=L_{\ii}(r,\sigma_{n+1})+1}^{L_{\ii}(r,\sigma_n)}\, p_n^{\sigma}(\Pi_n^{\sigma} i_{\ell}).
\end{equation*}
The last equality in the assertion follows from definition~\eqref{eq:11} of $P_{n-1}^{\sigma}(\Pi_{n}^{\sigma} i_{\ell})$.
\end{proof}

\begin{rem}\label{rem:1}
Assume $B_{\ii}(r)$ is $\sigma$-ordered and $L_{\ii}(r,\sigma_m)=L_{\ii}(r,\sigma_{m-1})=\ldots=L_{\ii}(r,\sigma_{m-k})$ for some $1\leq k< m\leq d$. Then the formula for $\mu_{\mathbf{p}}(B_{\ii}(r))$ can also be calculated using the ordering $(\sigma_1,\ldots,\sigma_{m-k-1},\omega,\sigma_{m+1},\ldots,\sigma_d)$, where the first block is empty if $k=m-1$, the last block is empty if $m=d$ and $\omega$ is any permutation of $\sigma_m,\sigma_{m-1},\ldots,\sigma_{m-k}$.
\end{rem}

Motivated by the definition of $\dim_{\mathrm{A}}\nu$, the goal is to bound the ratio $\mu_{\mathbf{p}}(B_{\ii}(R))/\mu_{\mathbf{p}}(B_{\ii}(r))$ for approximate cubes with different orderings. The first step is to consider when $B_{\ii}(R)$ and $B_{\ii}(r)$ have the same ordering. 

Let $\lambda_{\min}\coloneqq\min_{n,i} \lambda_i^{(n)}$ and fix $\sigma\in\mathcal{A}$. For $1\leq n\leq d$, we introduce
\begin{equation}\label{eq:25}
\overline{k}_n^{\sigma} \coloneqq \argmax_{i\in\mathcal{I}_n^{\sigma}} \frac{ \log P_{n-1}^{\sigma}(i) }{ \log \lambda_{i}^{(\sigma_n)} },\qquad 
\underline{k}_n^{\sigma} \coloneqq \argmin_{i\in\mathcal{I}_n^{\sigma}} \frac{ \log P_{n-1}^{\sigma}(i) }{ \log \lambda_{i}^{(\sigma_n)} }
\end{equation}
and
\begin{equation}\label{eq:26}
\overline{s}_n^{\sigma} \coloneqq \frac{ \log P_{n-1}^{\sigma}\big(\overline{k}_n^{\sigma}\big) }{ \log \lambda_{\overline{k}_n^{\sigma}}^{(\sigma_n)} },\qquad
\underline{s}_n^{\sigma} \coloneqq \frac{ \log P^\sigma_{n-1}\big(\underline{k}_n^{\sigma}\big) }{ \log \lambda_{\underline{k}_n^{\sigma}}^{(\sigma_n)} }.
\end{equation}
With this notation $\overline{S}(\mathbf{p},\sigma) = \sum_{n=1}^d \overline{s}_n^{\sigma}$ and $\underline{S}(\mathbf{p},\sigma) = \sum_{n=1}^d \underline{s}_n^{\sigma}$. If there are multiple choices for either $\overline{k}_n^{\sigma}$ or $\underline{k}_n^{\sigma}$, then choose one arbitrarily.

\begin{lma}\label{lem:ratioOfMeasures}
Fix $\sigma\in\mathcal{A}$ and assume that both $B_{\ii}(R)$ and $B_{\ii}(r)$ are $\sigma$-ordered, where $0<R\leq 1$ and  $r<\lambda_{\min} R$. Then there exists a constant $C>1$ depending only on the sponge $F$ such that 
\begin{equation*}
C^{-1} \left( \frac{R}{r} \right)^{ \underline{S}(\mathbf{p},\sigma) }
\leq \frac{\mu_{\mathbf{p}}(B_{\ii}(R))}{\mu_{\mathbf{p}}(B_{\ii}(r))}
\leq C \left( \frac{R}{r} \right)^{ \overline{S}(\mathbf{p},\sigma) }.
\end{equation*}
\end{lma}
\begin{proof}
It follows from Lemma~\ref{lem:measureApproxCube} that
\begin{equation*}
	\frac{\mu_{\mathbf{p}}(B_{\ii}(R))}{\mu_{\mathbf{p}}(B_{\ii}(r))}
	= \prod_{n=1}^{d}\, \prod_{\ell=L_{\ii}(R,\sigma_n)+1}^{L_{\ii}(r,\sigma_n)} \frac{1}{P_{n-1}^{\sigma}(\Pi_{n}^{\sigma}i_{\ell})} 
	= \prod_{n=1}^{d}\, \prod_{\ell=L_{\ii}(R,\sigma_n)+1}^{L_{\ii}(r,\sigma_n)} \big( \lambda_{\Pi_n^\sigma i_{\ell}}^{(\sigma_n)} \big)^{\frac{\log P_{n-1}^{\sigma}(\Pi_{n}^{\sigma}i_{\ell})}{ -\log \lambda_{\Pi_n^{\sigma} i_{\ell}}^{(\sigma_n)} }}.
\end{equation*}
The requirement that $r<\lambda_{\min} R$ ensures that $L_{\ii}(R,\sigma_n)<L_{\ii}(r,\sigma_n)$ for all $n$. We bound each exponent individually to obtain
\begin{equation}\label{eq:24}
\prod_{n=1}^{d} \left( \prod_{\ell=L_{\ii}(R,\sigma_n)+1}^{L_{\ii}(r,\sigma_n)}  \lambda_{\Pi_n^{\sigma} i_{\ell}}^{(\sigma_n)}  \right)^{ -\underline{s}_n^{\sigma} }
\leq \frac{\mu_{\mathbf{p}}(B_{\ii}(R))}{\mu_{\mathbf{p}}(B_{\ii}(r))}
\leq \prod_{n=1}^{d} \left( \prod_{\ell=L_{\ii}(R,\sigma_n)+1}^{L_{\ii}(r,\sigma_n)}  \lambda_{\Pi_n^{\sigma} i_{\ell}}^{(\sigma_n)}  \right)^{ -\overline{s}_n^{\sigma} }.
\end{equation}
From definition~\eqref{eq:21} of $L_{\ii}(r,n)$ it follows that there exists $C>1$ such that
\begin{equation}\label{eq:27}
C^{-1} \cdot\frac{r}{R}\leq \prod_{\ell=L_{\ii}(R,n)+1}^{L_{\ii}(r,n)}  \lambda_{i_{\ell}}^{(n)} \leq C\cdot \frac{r}{R},
\end{equation}
which together with~\eqref{eq:24} concludes the proof.
\end{proof}

Now we extend Lemma~\ref{lem:ratioOfMeasures} so that $B_{\ii}(R)$ and $B_{\ii}(r)$ can have different orderings. This step is not necessary if $F$ is a Lalley--Gatzouras sponge and represents one of the key technical challenges in the paper.

\begin{prop}\label{prop:ratioOfMeasures}
Assume $0<R\leq 1$ and  $r<\lambda_{\min} R$. Then there exists a constant $C>1$ depending only on the sponge $F$ such that 
\begin{equation*}
	C^{-1} \left( \frac{R}{r} \right)^{ \min_{\sigma\in\mathcal{A}} \underline{S}(\mathbf{p},\sigma) }
	\leq \frac{\mu_{\mathbf{p}}(B_{\ii}(R))}{\mu_{\mathbf{p}}(B_{\ii}(r))}
	\leq C \left( \frac{R}{r} \right)^{ \max_{\sigma\in\mathcal{A}} \overline{S}(\mathbf{p},\sigma) }.
\end{equation*}
\end{prop}

\subsection{Proof of Proposition~\ref{prop:ratioOfMeasures}}

Let $\sigma_{\ii}(r)$ denote the ordering of $B_{\ii}(r)$ and assume $\sigma_{\ii}(R)\neq \sigma_{\ii}(r)$. Trying to estimate the ratio $\mu_{\mathbf{p}}(B_{\ii}(R))/\mu_{\mathbf{p}}(B_{\ii}(r))$ directly using Lemma~\ref{lem:measureApproxCube} did not lead us to a proof. Instead, the rough idea is to divide the interval $[r,R]$ of scales into a uniformly bounded number of subintervals so that the ordering at roughly the two endpoints of a subinterval are the same. Then we repeatedly apply Lemma~\ref{lem:ratioOfMeasures} to each subinterval. The next lemma allows us to make a subdivision. 

\begin{lma}\label{lem:40}
	Fix $\varepsilon>0$ such that $1-\varepsilon> \max_{n,i} \lambda_i^{(n)}$. There exists a constant $C_1=C_1(F,\mathbf{p},\varepsilon)<\infty$ such that for all $\ii\in\Sigma$ and $0<R\leq 1$,
	\begin{equation*}
		\frac{\mu_{\mathbf{p}}(B_{\ii}(R))}{\mu_{\mathbf{p}}\big(B_{\ii}((1-\varepsilon)R)\big)} \leq C_1.
	\end{equation*}
\end{lma} 
\begin{proof}
First assume that $\sigma_{\ii}(R)= \sigma_{\ii}((1-\varepsilon)R)=\sigma$ and consider the symbolic representation of $B_{\ii}(R)$ and $B_{\ii}((1-\varepsilon)R)$. They could be different at indices 
\begin{equation*}
	L_{\ii}(R,\sigma_n)+1,\ldots, L_{\ii}(R,\sigma_n)+ d-n+1 \quad\text{ for each } 1\leq n\leq d,
\end{equation*}
but necessarily agree at all other indices due to the choice of $\varepsilon$. Where they agree, the corresponding terms simply cancel out in $\mu_{\mathbf{p}}(B_{\ii}(R)) / \mu_{\mathbf{p}}\big(B_{\ii}((1-\varepsilon)R)\big)$.  Hence, there are at most $1+2+\ldots+d<d^2$ different indices of interest. An index where they differ corresponds in $\mu_{\mathbf{p}}(B_{\ii}(R)) / \mu_{\mathbf{p}}\big(B_{\ii}((1-\varepsilon)R)\big)$ to a ratio $p/q$, where $p\geq q$ ($p$ is a sum containing $q$ by~\eqref{eq:23}) and both $p,q$ are uniformly bounded away from $0$ (simply because $p$ and $q$ are sums of different terms of $\mathbf{p}$ which all are strictly positive to begin with).   Therefore, there exists a uniform upper bound $C$ for $p/q$. As a result,
\begin{equation*}
	\frac{\mu_{\mathbf{p}}(B_{\ii}(R))}{\mu_{\mathbf{p}}\big(B_{\ii}((1-\varepsilon)R)\big)} \leq C^{d^2},
\end{equation*}
completing the proof in this case by setting $C_1=C^{d^2}$.
	
We claim that even if $\sigma_{\ii}(R)\neq\sigma_{\ii}((1-\varepsilon)R)$, there still exists an ordering $\omega$ such that the value of $\mu_{\mathbf{p}}(B_{\ii}(R))$ is the same when calculating it with $\sigma_{\ii}(R)$ or $\omega$ and likewise, the value of $\mu_{\mathbf{p}}(B_{\ii}((1-\varepsilon)R))$ is the same when calculating it with $\sigma_{\ii}((1-\varepsilon)R)$ or $\omega$. Hence, we may apply the previous argument to $\omega$.

To see the claim, first observe that the ordering $\sigma_{\ii}(R)=\sigma$ can be partitioned into $1\leq K\leq d$ blocks along indices $d\geq n_1>n_2>\ldots>n_K=1$ so that
\begin{multline*}
L_{\ii}(R,\sigma_d) = \ldots = L_{\ii}(R,\sigma_{n_1}) < L_{\ii}(R,\sigma_{n_1+1}) = \ldots = L_{\ii}(R,\sigma_{n_2}) \\
<\ldots<  L_{\ii}(R,\sigma_{n_{K-1}+1}) = \ldots = L_{\ii}(R,\sigma_{n_K}).
\end{multline*} 
The assumption on $\varepsilon$ implies that $L_{\ii}((1-\varepsilon)R,n)-L_{\ii}(R,n)\in\{0,1\}$ for all coordinates $n$. Therefore, we can partition the block $X_{\ell}\coloneqq\{\sigma_{n_{\ell}},\ldots,\sigma_{n_{\ell-1}+1}\}$ into $Y_{\ell}\sqcup Z_{\ell}$, where $Y_{\ell}=\{n\in X_{\ell}:\, L_{\ii}((1-\varepsilon)R,n)=L_{\ii}(R,n)\}$ and $Z_{\ell}=\{n\in X_{\ell}:\, L_{\ii}((1-\varepsilon)R,n)=L_{\ii}(R,n)+1\}$. We fix an (arbitrary) ordering of all $Y_{\ell}$, $Z_{\ell}$ and define
\begin{equation*}
\omega\coloneqq \big(Z_K,Y_K,Z_{K-1},Y_{K-1},\ldots,Z_2,Y_2,Z_1,Y_1\big).
\end{equation*}
By Remark~\ref{rem:1}, $\mu_{\mathbf{p}}(B_{\ii}(R))$ can be calculated by another ordering that only permutes elements within any of the blocks $X_{\ell}$. The ordering $\omega$ clearly satisfies this. It remains to argue that $\omega$ also works for $\sigma_{\ii}((1-\varepsilon)R)$. 

If $n\in Y_{\ell}$ and $m\in Z_\ell$ (for some $1\leq \ell\leq K$) then by definition $L_{\ii}((1-\varepsilon)R,n) < L_{\ii}((1-\varepsilon)R,m)$. Moreover, If $n\in Z_{\ell}$ and $m\in Y_{\ell+1}$ (for some $1\leq \ell\leq K-1$) then we also have $L_{\ii}((1-\varepsilon)R,n) \leq L_{\ii}((1-\varepsilon)R,m)$. These imply that $\sigma_{\ii}((1-\varepsilon)R)$ can be obtained from $\omega$ by only permuting elements within a block $Y_{\ell}$ or $Z_{\ell}$. This exactly means that $\mu_{\mathbf{p}}\big(B_{\ii}((1-\varepsilon)R)\big)$ can be calculated using $\omega$.
\end{proof}

We next define the scales where we subdivide $[r,R]$. Let
\begin{equation*}
	R_1\coloneqq \inf\{ r'>r:\, \sigma_{\ii}(r')=\sigma_{\ii}(R) \}
\end{equation*}
and terminate if $\sigma_{\ii}((1-\varepsilon)R_{1})=\sigma_{\ii}(r)$, otherwise, for $k\geq 2$ until $\sigma_{\ii}((1-\varepsilon)R_{k})=\sigma_{\ii}(r)$ define
\begin{equation*}
	R_k\coloneqq \inf\{ r'>r:\, \sigma_{\ii}(r') = \sigma_{\ii}((1-\varepsilon)R_{k-1}) \}
\end{equation*}
concluding with $R_M$, where $\varepsilon=\varepsilon(r)>0$ is chosen so small that $1-\varepsilon>\max\{r/R_M, \max_{n,i} \lambda_i^{(n)}\}$. It follows from the construction that $\sigma_{\ii}((1-\varepsilon)R_{k})$ is always different from the previous orderings, hence, $M\leq d!$.

We are ready to conclude the proof. We suppress multiplicative constants $c$ depending only on $F$ by writing $X\lesssim Y$ if $X\leq cY$. First using Lemma~\ref{lem:40} and then Lemma~\ref{lem:ratioOfMeasures}, we get the upper bound
\begin{align*}
	\frac{\mu_{\mathbf{p}}(B_{\ii}(R))}{\mu_{\mathbf{p}}\big(B_{\ii}(r)\big)} &\lesssim  \frac{\mu_{\mathbf{p}}(B_{\ii}(R))}{\mu_{\mathbf{p}}\big(B_{\ii}(R_1)\big)} \cdot
	\prod_{k=2}^M \frac{\mu_{\mathbf{p}}(B_{\ii}((1-\varepsilon)R_{k-1}))}{\mu_{\mathbf{p}}\big(B_{\ii}(R_k)\big)} \cdot
	\frac{\mu_{\mathbf{p}}(B_{\ii}((1-\varepsilon)R_M))}{\mu_{\mathbf{p}}\big(B_{\ii}(r)\big)} \\
	&\lesssim \left( \frac{R}{R_1} \right)^{ \overline{S}(\mathbf{p},\sigma_{\ii}(R_1)) } 
	\prod_{k=2}^M \left( \frac{(1-\varepsilon)R_{k-1}}{R_k} \right)^{ \overline{S}(\mathbf{p},\sigma_{\ii}(R_k)) } 
	\left( \frac{(1-\varepsilon)R_M}{r} \right)^{ \overline{S}(\mathbf{p},\sigma_{\ii}(r)) } \\
	&\lesssim \left( \frac{R}{r} \right)^{ \max_{\sigma\in\mathcal{A}} \overline{S}(\mathbf{p},\sigma) }.
\end{align*}
The lower bound is very similar. Lemma~\ref{lem:40} is not necessary because $\mu_{\mathbf{p}}(B_{\ii}(R)) \geq \mu_{\mathbf{p}}\big(B_{\ii}((1-\varepsilon)R)\big)$ holds for any $R>0$ and one uses $\min_{\sigma\in\mathcal{A}} \underline{S}(\mathbf{p},\sigma)$ instead in the last step. The proof of Proposition~\ref{prop:ratioOfMeasures} is complete.

\section{Proof of Theorem~\ref{thm:dimMeasure}}\label{sec:proofofmain}

\subsection{Transferring symbolic estimates to geometric estimates}

The very strong SPPC implies that there exists $\delta_0>0$ depending only on the sponge $F$ such that for every $\sigma\in\mathcal{A}$, $1\leq n\leq d$ and $i,j\in\mathcal{I}$ for which $f_i$ and $f_j$ do \emph{not} overlap exactly on $E_n^{\sigma}$,
\begin{equation*}
\mathrm{dist}\big( \Pi_n^{\sigma}\big(f_i([0,1]^d)\big), \Pi_n^{\sigma}\big(f_j([0,1]^d)\big) \big) \geq \delta_0.
\end{equation*}
The next lemma allows us to replace a Euclidean ball $B(x,r)$ with the image of an approximate cube of roughly the same diameter under the natural projection  $\pi$, recall~\eqref{eq:14}. It is an adaptation of~\cite[Proposition~6.2.1]{Olsen_PJM98}. The short proof is included for completeness.
\begin{lma}\label{lem:30}
Assume the sponge $F$ satisfies the SPPC. For all $\ii\in\Sigma$ and $r>0$,
\begin{equation*}
\pi(B_{\ii}(r)) \subseteq B(\pi(\ii),\sqrt{d}\cdot r).
\end{equation*}
Moreover, if $F$ satisfies the very strong SPPC, then
\begin{equation*}
B(\pi(\ii), \delta_0 \cdot r)\cap F \subseteq \pi(B_{\ii}(r)).
\end{equation*}
\end{lma}
\begin{proof}
By definition~\eqref{eq:20}, the image $\pi(B_{\ii}(r))$ is contained inside a cuboid of side length $\prod_{\ell=1}^{L_{\ii}(r,n)} \lambda_{i_\ell}^{(n)} \leq r$ in each coordinate $n$. In worst case, $\pi(\ii)$ is a corner of this cuboid which is then certainly contained in $B(\pi(\ii),\sqrt{d}\cdot r)$.

We show that $\jj\notin B_{\ii}(r)$ implies $\pi(\jj)\notin B(\pi(\ii), \delta_0 \cdot r)$ under the very strong SPPC. Assume $B_{\ii}(r)$ is $\sigma$-ordered. Since $\jj\notin B_{\ii}(r)$, there is a largest $n'\in\{1,\ldots,d\}$ and a smallest $\ell'\in\{L_{\ii}(r,\sigma_{n'+1})+1,\ldots,L_{\ii}(r,\sigma_{n'})\}$ such that $\Pi_{n'}^{\sigma}j_{\ell'}\neq \Pi_{n'}^{\sigma}i_{\ell'}$. For this particular choice, the very strong SPPC implies that
\begin{equation*}
\mathrm{dist}\big( \Pi_{n'}^{\sigma}\big(f_{i_{\ell'}}([0,1]^d)\big), \Pi_{n'}^{\sigma}\big(f_{j_{\ell'}}([0,1]^d)\big) \big) \geq \delta_0.
\end{equation*}
Since the projection $\Pi_{n'}^{\sigma}$ can only decrease distance and $\Pi_{n'}^{\sigma}j_{\ell}= \Pi_{n'}^{\sigma}i_{\ell}$ for all $\ell<\ell'$, we can bound
\begin{align*}
\mathrm{dist}(\pi(\ii),\pi(\jj)) &\geq \mathrm{dist}\big(\Pi_{n'}^{\sigma}(\pi(\ii)),\Pi_{n'}^{\sigma}(\pi(\jj))\big) \\
&\geq \mathrm{dist}\big(\Pi_{n'}^{\sigma}\big(f_{i_1\ldots i_{\ell'-1}i_{\ell'}}([0,1]^d)\big), \Pi_{n'}^{\sigma}\big(f_{j_1\ldots j_{\ell'-1}j_{\ell'}}([0,1]^d)\big)\big) \\
&\geq \delta_0\cdot \prod_{\ell=1}^{\ell'-1} \lambda_{i_\ell}^{(\sigma_{n'})} \\
&\stackrel{\eqref{eq:21}}{>} \delta_0\cdot r,
\end{align*}
completing the proof.
\end{proof}
The very strong  SPPC also implies that the natural projection $\pi$ is injective and that $\mu_{\mathbf{p}}(B_{\ii}(r))=\nu_{\mathbf{p}}(\pi(B_{\ii}(r)))$ for any approximate cube $B_{\ii}(r)$. This, together with Lemma~\ref{lem:30}, implies that for any $\ii\in\Sigma$ and $0<r<R\leq 1$,
\begin{equation*}
\frac{\mu_{\mathbf{p}}\big(B_{\ii}(R/\sqrt{d})\big)}{\mu_{\mathbf{p}}\big(B_{\ii}(r/\delta_0)\big)}
\leq \frac{\nu_{\mathbf{p}}(B(\pi(\ii),R))}{\nu_{\mathbf{p}}(B(\pi(\ii),r))} \leq 
\frac{\mu_{\mathbf{p}}\big(B_{\ii}(R/\delta_0)\big)}{\mu_{\mathbf{p}}\big(B_{\ii}(r/\sqrt{d})\big)}.
\end{equation*}
Hence, it is enough to consider the ratio $\mu_{\mathbf{p}}(B_{\ii}(R))/\mu_{\mathbf{p}}(B_{\ii}(r))$.

The upper bound for $\dim_{\mathrm{A}}\nu_{\mathbf{p}}$ and the lower bound for $\dim_{\mathrm{L}}\nu_{\mathbf{p}}$ now directly follow from Proposition~\ref{prop:ratioOfMeasures}.  The lower bound for $\dim_{\mathrm{A}}\nu_{\mathbf{p}}$  requires more work and we give the argument in the next subsection.  The argument to bound $\dim_{\mathrm{L}}\nu_{\mathbf{p}}$ from above is analogous and we omit the details.

\subsection{Lower bound for Assouad dimension}

Recall notation from~\eqref{eq:25} and~\eqref{eq:26}.  Using Lemma \ref{lem:30}, in order to bound $\dim_{\mathrm{A}}\nu_{\mathbf{p}}$ from below, it suffices to prove the following proposition.
\begin{prop}
For all $\sigma \in \mathcal{B}$, there exists  a sequence of triples $(R, r, \ii) \in (0,1) \times (0,1) \times \Sigma$ with     $R/r\to\infty$ such that 
\begin{equation*}
\frac{\mu_{\mathbf{p}}(B_{\ii}(R))}{\mu_{\mathbf{p}}(B_{\ii}(r))} \geq c \left( \frac{R}{r} \right)^{ \overline{S}(\mathbf{p},\sigma) }
\end{equation*}
for some constant $c>0$ uniformly for all triples in the sequence.
\end{prop}

\begin{proof}
Fix $\sigma \in \mathcal{B}$.  By definition of $\mathcal{B}$, there exist  $\jj \in\Sigma$  and  $\rho>0$  such that  $\jj$ determines a strictly $\sigma$-ordered cylinder at scale $\rho$, see \eqref{eq:201}.  By passing to a finite iterate of the original IFS if necessary, we may assume that $L_{\jj}(\rho,\sigma_d) = 1$.  Therefore, there exists $j \in \mathcal{I}$ with
\[
 \lambda_{j}^{(\sigma_d)} <  \lambda_{j}^{(\sigma_{d-1})} < \ldots <   \lambda_{j}^{(\sigma_1)}.
\]
We use this $j$  together with $\overline{k}^\sigma_n$ ($n=1, \dots, d)$, see \eqref{eq:25}, to build  the $\ii\in\Sigma$ from the statement of the proposition.  We now construct the sequence of triples $(R,r,\ii)$.  This is done by first choosing a decreasing sequence of $R$ tending to 0 with the first term sufficiently small.  For a particular $R$ the associated $r$ and $\ii$ are built as follows.  First, for $a,b \in \{1, \dots, d\}$ and $v \in \mathcal{I}$, write 
\begin{equation} \label{thetanot}
\theta_{b}^{a}(v)  = \frac{\log \lambda_v^{(a)}}{\log \lambda_v^{(b)}} >0.
\end{equation}
Observe that
\begin{equation} \label{goodoldj}
\theta_{\sigma_{n}}^{\sigma_{n-1}}(j) < 1
\end{equation}
for all $n=2, \dots, d$. Choose $r$ to satisfy
\begin{equation} \label{keychoice}
\max_{v \in \mathcal{I}, n=2, \dots, d} \frac{R^{1+\frac{1-\theta_{\sigma_{n}}^{\sigma_{n-1}}(j)}{\sum_{\ell=d}^n\theta_{\sigma_{\ell}}^{\sigma_{n-1}}(v) }}}{\lambda_{\min}^{1+\frac{\theta_{\sigma_{n}}^{\sigma_{n-1}}(j)}{\sum_{\ell=d}^n\theta_{\sigma_{\ell}}^{\sigma_{n-1}}(v) }}} < r< \lambda_{\min}  R .
\end{equation}
 Choosing $r$ in the range \eqref{keychoice} is possible for sufficiently small $R$ since the bound on the left is $o(R)$ as $R \to 0$ due to \eqref{goodoldj}.  Moreover, this allows the choice to be made whilst also ensuring  $R/r\to\infty$.   Let $\ii = i_1, i_2, \dots \in \Sigma$ be such that
\[
i_{\ell}  = \overline{k}^\sigma_n
\]
for $\ell=L_\ii(R,\sigma_n)+1, \dots, L_\ii(r,\sigma_n)$ and all other entries are $j$.  Note that the upper bound from \eqref{keychoice}   immediately guarantees
\[
L_\ii(R,\sigma_n)<L_\ii(r,\sigma_n)
\]
for all $n=1, \dots, d$.  In order to show that $\ii$ is indeed well-defined, we claim that the  lower bound from \eqref{keychoice}     guarantees
\begin{equation} \label{keyordering}
L_\ii(r,\sigma_n)<L_\ii(R,\sigma_{n-1})
\end{equation}
for $n=2, \dots, d+1$, where we adopt the convention that $L_\ii(r,\sigma_{d+1})= 0$.   This takes more work. Proceeding by (backwards)  induction let  $n \in \{2, \dots, d\}$ and  assume that \eqref{keyordering} holds for $ n+1, \dots, d+1$.  The goal is to establish \eqref{keyordering} for $n$.  By definition of $L_\ii(R, \sigma_{n-1})$, see \eqref{eq:21},
\[
\lambda_{\min}  R  < \prod_{\ell=1}^{L_\ii(R,  \sigma_{n-1})} \lambda_{i_{\ell}}^{( \sigma_{n-1})} \leq R
\]
and, therefore,  \eqref{keyordering} will hold provided
\[
\Lambda:=  \prod_{\ell=1}^{L_\ii(r,  \sigma_{n})} \lambda_{i_{\ell}}^{( \sigma_{n-1})}  > R.
\]
By the inductive hypothesis and construction of $\ii$, 
\[
\Lambda = \prod_{\ell=d}^n \left(\lambda_{j}^{( \sigma_{n-1})} \right)^{L_\ii(R,\sigma_{\ell}) - L_\ii(r,\sigma_{\ell+1}) }  \left(\lambda_{\overline{k}^\sigma_{\ell}}^{ (\sigma_{n-1})} \right)^{L_\ii(r,\sigma_{\ell}) - L_\ii(R,\sigma_{\ell}) },
\]
where we have changed the use of the index  $\ell$ slightly. Invoking \eqref{thetanot} and using \eqref{eq:21}
\begin{align*}
\Lambda &= \prod_{\ell=d}^n \left(\lambda_{j}^{( \sigma_{n})} \right)^{\theta_{\sigma_{n}}^{\sigma_{n-1}}(j) \left(L_\ii(R,\sigma_{\ell}) - L_\ii(r,\sigma_{\ell+1})\right) }  \left(\lambda_{\overline{k}^\sigma_{\ell}}^{ (\sigma_{\ell})} \right)^{\theta_{\sigma_{\ell}}^{\sigma_{n-1}}(\overline{k}^\sigma_{\ell}) \left(L_\ii(r,\sigma_{\ell}) - L_\ii(R,\sigma_{\ell})\right) } \\
&= \left(\prod_{\ell=d}^n \left(\lambda_{j}^{( \sigma_{n})} \right)^{L_\ii(R,\sigma_\ell) - L_\ii(r,\sigma_{\ell+1}) }  \right)^{\theta_{\sigma_{n}}^{\sigma_{n-1}}(j)}
\prod_{\ell=d}^n  \left(
\frac{\prod_{\ell=1}^{L_\ii(r,\sigma_{\ell})} \lambda_{i_{\ell}}^{(\sigma_{n-1})}}{\prod_{\ell=1}^{ L_\ii(R,\sigma_{\ell})} \lambda_{i_{\ell}}^{(\sigma_{n-1})}}
 \right)^{\theta_{\sigma_{\ell}}^{\sigma_{n-1}}(\overline{k}^\sigma_{\ell}) } \\
&\geq \left(\prod_{\ell=1}^{L_\ii(R,\sigma_n)} \lambda_{i_{\ell}}^{(\sigma_n)} \right)^{\theta_{\sigma_{n}}^{\sigma_{n-1}}(j)}\prod_{\ell=d}^n  \left(\frac{\lambda_{\min}r}{R} \right)^{\theta_{\sigma_{\ell}}^{\sigma_{n-1}}(\overline{k}^\sigma_{\ell}) } \\
&\geq  (\lambda_{\min}R)^{\theta_{\sigma_{n}}^{\sigma_{n-1}}(j)}
  \left(\lambda_{\min} \frac{r}{R} \right)^{\sum_{\ell=d}^n\theta_{\sigma_{\ell}}^{\sigma_{n-1}}(\overline{k}^\sigma_{\ell}) } \\
&>R
\end{align*}
by \eqref{keychoice}, which proves \eqref{keyordering}.   With  the sequence now in place, the result follows easily.  For all triples $(R,r,\ii)$  in the sequence,   Lemma~\ref{lem:measureApproxCube}  gives
\begin{align*}
	\frac{\mu_{\mathbf{p}}(B_{\ii}(R))}{\mu_{\mathbf{p}}(B_{\ii}(r))}
	&= \prod_{n=1}^{d}\, \prod_{\ell=L_{\ii}(R,\sigma_n)+1}^{L_{\ii}(r,\sigma_n)} \frac{1}{P_{n-1}^{\sigma}( i_{\ell})} \\
	&= \prod_{n=1}^{d}\, \prod_{\ell=L_{\ii}(R,\sigma_n)+1}^{L_{\ii}(r,\sigma_n)} \big( \lambda_{  i_{\ell}}^{(\sigma_n)} \big)^{-\overline{s}_n^\sigma} \qquad \text{(by construction of $\ii$)}\\
	&= \prod_{n=1}^{d}\,  \left( \frac{\prod_{\ell=1}^{L_{\ii}(r,\sigma_n)}  \lambda_{  i_{\ell}}^{(\sigma_n)} }
{\prod_{\ell=1}^{L_{\ii}(R,\sigma_n)} \lambda_{  i_{\ell}}^{(\sigma_n)} }
 \right)^{-\overline{s}_n^\sigma} \\
	&\geq  \prod_{n=1}^{d}\,  \left( \frac{\lambda_{\min} R}{r}
 \right)^{\overline{s}_n^\sigma} \\
& = \lambda_{\min}^{ \overline{S}(\mathbf{p},\sigma)}  \left( \frac{R}{r} \right)^{ \overline{S}(\mathbf{p},\sigma) }
\end{align*}
as required. 
\end{proof}

\subsection{Final claim for $\sigma$-ordered coordinate-wise natural measures}

The claim for the $\sigma$-ordered coordinate-wise natural measure $\mathbf{q}^{\sigma}$ follows from the simple observation that 
\begin{equation*}
	Q_{n-1}^{\sigma}(\Pi_{n}^{\sigma}i)
	= \frac{q_n^{\sigma}(\Pi_n^{\sigma} i)}{q_{n-1}^{\sigma}(\Pi_{n-1}^{\sigma} i)}
	= \frac{ \prod_{m=1}^n \big( \lambda_{\Pi_m^{\sigma} i}^{(\sigma_m)} \big)^{s_{m-1}^{\sigma}(\Pi_{m-1}^{\sigma} i)} }{ \prod_{m=1}^{n-1} \big( \lambda_{\Pi_m^{\sigma} i}^{(\sigma_m)} \big)^{s_{m-1}^{\sigma}(\Pi_{m-1}^{\sigma} i)} }
	= \big( \lambda_{\Pi_n^{\sigma} i}^{(\sigma_n)} \big)^{s_{n-1}^{\sigma}(\Pi_{n-1}^{\sigma} i)}.
\end{equation*}
Hence,
\begin{equation*}
	\frac{ \log Q_{n-1}^{\sigma}(i) }{ \log \lambda_{i}^{(\sigma_n)} } = s_{n-1}^{\sigma}(\Pi_{n-1}^{\sigma}i) \;\text{ for every } i\in\mathcal{I}_n^{\sigma} \text{ and } 1\leq n\leq d,
\end{equation*}
completing the proof of the claim.

\begin{center} \textbf{Acknowledgements}
\end{center}
Both authors were  financially supported by a \textit{Leverhulme Trust Research Project Grant} (RPG-2019-034). JMF was also financially supported by  an \textit{EPSRC Standard Grant} (EP/R015104/1) and an \textit{RSE Sabbatical Research Grant} (70249).

\bibliographystyle{abbrv}
\bibliography{biblio_LGSponges}

\begin{thebibliography}{10}

\bibitem{BARANSKIcarpet_2007}
K.~Bara{\'n}ski.
\newblock Hausdorff dimension of the limit sets of some planar geometric
  constructions.
\newblock {\em Adv. Math.}, 210(1):215--245, 2007.

\bibitem{BJK_Chaos_IMRN22}
B.~B\'ar\'any, N.~Jurga, and I.~Kolossv\'ary.
\newblock {On the Convergence Rate of the Chaos Game}.
\newblock {\em Int. Math. Res. Not. IMRN}, 2022.
\newblock rnab370.

\bibitem{Bedford84_phd}
T.~Bedford.
\newblock {\em Crinkly curves, Markov partitions and box dimensions in
  self-similar sets}.
\newblock PhD thesis, University of Warwick, 1984.

\bibitem{BylundGudayol_ProcAMS00}
P.~Bylund and J.~Gudayol.
\newblock On the existence of doubling measures with certain regularity
  properties.
\newblock {\em Proc. Amer. Math. Soc.}, 128(11):3317--3327, 2000.

\bibitem{das_fishman_simmons_urbanski_2019ETDS}
T.~Das, L.~Fishman, D.~Simmons, and M.~Urba\'nski.
\newblock Badly approximable points on self-affine sponges and the lower
  {A}ssouad dimension.
\newblock {\em Ergodic Theory Dynam. Systems}, 39(3):638--657, 2019.

\bibitem{das2017hausdorff}
T.~Das and D.~Simmons.
\newblock The {H}ausdorff and dynamical dimensions of self-affine sponges: a
  dimension gap result.
\newblock {\em Invent. Math.}, 210(1):85--134, 2017.

\bibitem{FalconerBook}
K.~J. Falconer.
\newblock {\em Fractal Geometry: Mathematical Foundations and Applications}.
\newblock 3rd Ed., John Wiley {\&} Sons, Hoboken, NJ, 2014.

\bibitem{FFK_MinkowskiDimMeas_2020arxiv}
K.~J. Falconer, J.~M. Fraser, and A.~Käenmäki.
\newblock Minkowski dimension for measures.
\newblock {\em arXiv e-prints}, 2001.07055, 2020.
\newblock to appear in Proc. Amer. Math. Soc.

\bibitem{FengWang2005}
D.-J. Feng and Y.~Wang.
\newblock A class of self-affine sets and self-affine measures.
\newblock {\em J. Fourier Anal. Appl.}, 11(1):107--124, 2005.

\bibitem{Fraser_TAMS2014}
J.~M. Fraser.
\newblock Assouad type dimensions and homogeneity of fractals.
\newblock {\em Trans. Amer. Math. Soc.}, 366(12):6687--6733, 2014.

\bibitem{fraser_2020Book}
J.~M. Fraser.
\newblock {\em Assouad Dimension and Fractal Geometry}.
\newblock Number 222 in Cambridge Tracts in Mathematics. Cambridge University
  Press, 2020.

\bibitem{FraserHowroyd_AnnAcadSciFennMath17}
J.~M. Fraser and D.~Howroyd.
\newblock Assouad type dimensions for self-affine sponges.
\newblock {\em Ann. Acad. Sci. Fenn. Math.}, 42:149--174, 2017.

\bibitem{FraserHowroyd_Indiana20}
J.~M. Fraser and D.~Howroyd.
\newblock On the upper regularity dimensions of measures.
\newblock {\em Indiana Univ. Math. J.}, 69(2):685--712, 2020.

\bibitem{GatzourasLalley92}
D.~Gatzouras and S.~P. Lalley.
\newblock Hausdorff and box dimensions of certain self-affine fractals.
\newblock {\em Indiana Univ. Math. J.}, 41(2):533--568, 1992.

\bibitem{Howroyd_JFG19}
D.~C. Howroyd.
\newblock Assouad type dimensions for self-affine sponges with a weak
  coordinate ordering condition.
\newblock {\em J. Fractal Geom.}, 6(1):67--88, 2019.

\bibitem{KaenmakiLehrbackVuorinen_Indiana13}
A.~K\"{a}enm\"{a}ki, J.~Lehrb\"{a}ck, and M.~Vuorinen.
\newblock Dimensions, {W}hitney covers, and tubular neighborhoods.
\newblock {\em Indiana Univ. Math. J.}, 62:1861--1889, 2013.

\bibitem{KenyonPeres_ETDS96}
R.~Kenyon and Y.~Peres.
\newblock Measures of full dimension on affine-invariant sets.
\newblock {\em Ergodic Theory Dynam. Systems}, 16(2):307--323, 1996.

\bibitem{King_LocalDimBMCarpet_95AdvMath}
J.~King.
\newblock The singularity spectrum for general {S}ierpi\'nski carpets.
\newblock {\em Adv. Math.}, 116(1):1--11, 1995.

\bibitem{LuukkainenSaksman_ProcAMS98}
J.~Luukkainen and E.~Saksman.
\newblock Every complete doubling metric space carries a doubling measure.
\newblock {\em Proc. Amer. Math. Soc.}, 126(2):531--534, 1998.

\bibitem{mcmullen84}
C.~McMullen.
\newblock The {H}ausdorff dimension of general {S}ierpi{\'n}ski carpets.
\newblock {\em Nagoya Math. J.}, 96:1--9, 1984.

\bibitem{Olsen_PJM98}
L.~Olsen.
\newblock Self-affine multifractal {S}ierpi\'nski sponges in $\mathbb{R}^d$.
\newblock {\em Pacific J. Math.}, 183(1):143--199, 1998.

\bibitem{peressolomyaksurvey}
Y.~Peres and B.~Solomyak.
\newblock Problems on {S}elf-similar {S}ets and {S}elf-affine {S}ets: {A}n
  {U}pdate.
\newblock In C.~Bandt, S.~Graf, and M.~Z{\"a}hle, editors, {\em Fractal
  Geometry and Stochastics II}, pages 95--106. Birkh{\"a}user Basel, 2000.

\end{thebibliography}

\end{document}